\documentclass[a4paper,reqno,11pt]{amsart}
\usepackage{amsfonts,amsmath,amsthm,amssymb,stmaryrd}
\usepackage{hyperref}
\usepackage{color}
\usepackage{mathrsfs}
\usepackage{tikz}

\newtheorem{theorem}{Theorem}[section]
\newtheorem{lemma}{Lemma}[section]
\newtheorem{proposition}{Proposition}[section]
\newtheorem{corollary}{Corollary}[section]

\theoremstyle{definition}
\newtheorem{definition}{Definition}[section]

\theoremstyle{remark}
\newtheorem{remark}{Remark}[section]

\numberwithin{equation}{section}
\allowdisplaybreaks 
\begin{document}
\author{Chenlu Zhang}
\address{College of Mathematics and Statistics, Chongqing University,
                             Chongqing, 401331,  China.}
\email{20220601002@stu.cqu.edu.cn}
\author{Huaqiao Wang}
\address{College of Mathematics and Statistics, Chongqing University,
                             Chongqing, 401331,  China.}
\email{wanghuaqiao@cqu.edu.cn}

\title[Global well-posedness of the LLS equation]
{Global well-posedness of smooth solutions to the Landau--Lifshitz--Slonczewski equation}
\thanks{Corresponding author: wanghuaqiao@cqu.edu.cn}
\keywords{Landau--Lifshitz--Slonczewski equation, Smooth solutions, Existence and uniqueness, Morrey spaces, energy estimates.}\
\subjclass[2010]{82D40; 35K10; 35D35.}
\begin{abstract}
In this paper, we mainly consider the global solvability of smooth solutions for the Cauchy problem of the three-dimensional Landau--Lifshitz--Slonczewski equation in the Morrey space. We derive the covariant complex Ginzburg--Landau equation by using moving frames to address the nonlinear parts. Applying the semigroup estimates and energy methods, we extend local classical solutions to global solutions and prove the boundedness of $\|\nabla\boldsymbol{m}\|_{L^{\infty}(\mathbb{R}^{3})}$, where $\boldsymbol{m}$ is the magnetic intensity. Moreover, we obtain a global weak solution by using an approximation result and improve the regularity of the obtained solution by the regularity theory. Finally, we establish the existence and uniqueness of global smooth solutions under some conditions on $\nabla\boldsymbol{m}_{0}$ and the density of the spin-polarized current.
\end{abstract}

\maketitle
\section{Introduction}\label{Sec1}
After the discovery of the giant magnetoresistance effect, the spin-transfer torque effect (STT for short) is another milestone in the theory of micromagnetism. In the 1970s, Berger \cite{BerL,Freitas,Hung} discovered that electric currents can drive the motion of magnetic domain walls. In the late 1980s, Slonczewski \cite{SlonJC} showed the existence of interlayer exchange coupling between the two ferromagnetic electrodes of a magnetic tunnel junction in theory. Until 1996, Slonczewski and Berger \cite{BerL2,SlonJC2} both proposed the STT effect in the structure of ``sandwich'' magnetic thin film, respectively. This important conclusion attracted widespread attention in the theory of micromagnetism, then the existence of the STT effect was be confirmed, for more details see \cite{Kub,San,Tso}. There are different types of STT, we are mainly concerned with the adiabatic and non-adiabatic STT. When the current flows along the film surface in any direction, they are represented by \cite{Smith}:
\begin{align*}
\boldsymbol{T}_{ad}&=\theta_{1}\boldsymbol{m}\times\left[\boldsymbol{m}\times(\boldsymbol{v}\cdot\nabla)\boldsymbol{m}\right],\\
\boldsymbol{T}_{na}&=\theta_{2}\boldsymbol{m}\times(\boldsymbol{v}\cdot\nabla)\boldsymbol{m},
\end{align*}
where $\boldsymbol{m}$ is the magnetic intensity, $\boldsymbol{v}$ denotes the density of the spin-polarized current along the current direction, $\theta_{1}$ and $\theta_{2}$ are dimensionless constants. $\boldsymbol{T}_{ad}$ denotes the adiabatic STT, which was first introduced by Bazaliy \cite{Ralph} when studying the ballistic transport of conduction electrons in half-metal. $\boldsymbol{T}_{na}$ denotes the non-adiabatic STT, which was proposed by Zhang \cite{ZhangLi} in exploiting non-equilibrium conduction electrons when the spin-polarized behavior deviates from its local magnetic torque.

Notice that if we do not consider the STT effect, the dynamic behavior of magnetization can be described by the Landau-Lifshitz-Gilbert (LLG for short) equation:
\begin{align}
\frac{\partial{\boldsymbol{m}}}{\partial{t}}=\theta_{3}\boldsymbol{m}\times\Delta\boldsymbol{m}+\theta_{4}
\left(\boldsymbol{m}\times\frac{\partial{\boldsymbol{m}}}{\partial{t}}\right),\;\; \boldsymbol{m}(0,x)=\boldsymbol{m}_0,\label{1.8}
\end{align}
where $\boldsymbol{m}(t,x):\mathbb{R}\times\mathbb{R}^{3}\rightarrow\mathbb{S}^{2}\subset \mathbb{R}^{3}$ represents the magnetic intensity. In this paper, we only consider the exchange field, namely $\Delta\boldsymbol{m}$ is the effective field. $\theta_{3}$ is the constant with respect to the magnetogyric ratio, $\theta_{4}$ is the Gilbert damping constant. The first term on the right-hand side of \eqref{1.8} represents Larmor precession and the second term is Gilbert damping term.

The current can produce the STT effect when the magnetic nanowires \cite{Hei,MyersAlbert} with in-plane current flow. Generally speaking, physicists \cite{Han} directly add the STT term to the right-hand side of \eqref{1.8}, then the LLG equation with the STT effect can be expressed by:
\begin{align}
\frac{{\partial} \boldsymbol{m}}{\partial t}&=\theta_{3}(\boldsymbol{m}\times\Delta\boldsymbol{m})
+\theta_{4}\left(\boldsymbol{m}\times\frac{{\partial}\boldsymbol{m}}{\partial t}\right)+\boldsymbol{T}_{ad}+\boldsymbol{T}_{na},\;\; \boldsymbol{m}(0,x)=\boldsymbol{m}_0. \label{1.5}
\end{align}
By using the fact that $|\boldsymbol{m}|=1$ and $\boldsymbol{m}\times\boldsymbol{m}\times (\boldsymbol{v}\cdot\nabla)\boldsymbol{m}=-(\boldsymbol{v}\cdot\nabla)\boldsymbol{m}$, adding $\theta_{4}\boldsymbol{m}\times \eqref{1.5}$ into \eqref{1.5}, for notational convenience, we neglect some coefficients and obtain the Landau-Lifshitz-Slonczewski (LLS for short) equation:
\begin{align}\label{1.6}
\begin{cases}
\frac{\partial\boldsymbol{m}}{\partial t}+(\boldsymbol{v}\cdot\nabla)\boldsymbol{m}
+\boldsymbol{m}\times (\boldsymbol{v}\cdot\nabla)\boldsymbol{m}
=-\boldsymbol{m}\times\Delta\boldsymbol{m}
-\varepsilon(\boldsymbol{m}\times\boldsymbol{m}\times\Delta\boldsymbol{m}),\\
\boldsymbol{m}(0,x)=\boldsymbol{m}_{0},
\end{cases}
\end{align}
where $\varepsilon\in(0,1)$ denotes the Gilbert damping constant. Taking into account the fact that $-\boldsymbol{m}\times\boldsymbol{m}\times\Delta\boldsymbol{m}=\Delta\boldsymbol{m}+|\nabla\boldsymbol{m}|^{2}\boldsymbol{m}$ by using $|\boldsymbol{m}|=1$, equation \eqref{1.6} can be written as
\begin{align}\label{main-eq}
\begin{cases}
\frac{\partial\boldsymbol{m}}{\partial t}+(\boldsymbol{v}\cdot\nabla)\boldsymbol{m}
+\boldsymbol{m}\times(\boldsymbol{v}\cdot\nabla)\boldsymbol{m}
=-\boldsymbol{m}\times\Delta \boldsymbol{m}+\varepsilon(\Delta\boldsymbol{m}+|\nabla\boldsymbol{m}|^{2}\boldsymbol{m}),\\
\boldsymbol{m}(0,x)=\boldsymbol{m}_{0}.
\end{cases}
\end{align}

When we ignore the STT effect, the LLS equation degenerate to the LLG equation. There have been many outstanding results for the existence of smooth solutions of the LLG equation. For the case of dimension $n=1$, Zhou-Guo-Tan applied the spatial difference method and the a prior estimate to prove the existence and uniqueness of smooth solutions for the initial-boundary value problem with periodic boundary condition, it is worthy that when Gilbert damping constant $\varepsilon\rightarrow0$, they obtained a unique global smooth solution, for details see \cite{ZhouGuoTan}. Later, Ding-Guo-Su \cite{DingGuoSu} gained local smooth solutions of the inhomogeneous LLG equation by the viscosity vanishing method. For the case of dimension $n=2$, Chen-Ding-Guo \cite{ChenDing} discussed that any weak solutions with finite energy are unique and smooth with the exception of at most finite points for the Cauchy problem. Inspired by \cite{ZhouGuoTan}, Guo-Hong \cite{GuoHong} considered the global smooth solution for a small initial data from $\mathbb{R}^{2}$ into the unit sphere $\mathbb{S}^{2}$. A few years later, Carbou-Fabrie \cite{Carbou} showed global existence of regular solutions for the initial-boundary problem value with small data. For the case of higher dimensions, Guo-Hong \cite{GuoHong} failed to directly get globally smooth solutions but obtained global weak solutions. Furthermore, in an unbounded Riemannian manifold, by applying approximation method, Ding-Wang \cite{DingWang} investigated that the short-time smooth solution blow-up apparently in finite time. Very recently, there are more results about the global smoothness of solutions. Based on the essential structure of equations, Melcher \cite{Melcher1} sought a new approach from the topology for transforming the LLG equation into the covariant LLG equation, and under a small condition of the initial data, they claimed that the Cauchy problem of the LLG equation has global smooth solutions in the Sobolev space. Later, Lin-Lai-Wang \cite{Lin} extended the method to the Morrey space and got the global solvability of the LLG equation.

When we pay attention to the regular solution, there is a significant approach that has been widely used to improve the regularity of weak solutions. The regularity theory was used widely in the heat flow for harmonic maps. For the case of dimension $m>2$, Chen-Struwe \cite{ChenStu} showed that there exists a global regular solution when the harmonic map $u_{0}$ is smooth by energy methods and monotonicity formula, and more similar results see \cite{Stu,Stu1}. For the case of dimension $m\geq 2$, Evans \cite{Evans} proved that the stationary Harmonic map exists a weak and smooth mapping in an open subset which satisfies the $(m-2)$-dimensional Hausdorff measure of its complementary set equals to zero. After, by using the parabolic monotonicity inequality, Morrey's lemma of parabolic type and BMO spaces, Feldman \cite{Feld} obtained the weak solution of the harmonic map of evolution is smooth away from an open set. A year later, Chen-Li-Lin \cite{ChenLiLin} established a similar result to Feldman, whereas there are substantial differences in some technical details, such as the application of the duality between the parabolic BMO and Hardy spaces. Inspired by the study of the heat flow for harmonic maps, there are some people considered the regularity theory of the LLG equation. In higher dimensions, Moser \cite{Moser} proved that there exists an open set with a complement of vanishing $n$-dimensional parabolic Hausdorff measure such that solutions are smooth by the stability condition and energy decay. Subsequently, Ding-Guo \cite{DingGuo,DingGuo1} successively exploited the partial regularity of weak solutions on a compact, smooth Riemannian manifold with boundary or without boundary. Later, for the case of smooth initial data with finite energy, Melcher \cite{Melcher} investigated the partial regularity of weak solutions by monotonicity inequality and suitable time slices method.

As we all know, the well-posedness of the LLG equation are very complete. It is easy to see that there are a few results about the well-posedness of the LLS equation. In Sobolev space, Melcher-Ptashnyk \cite{MelcherP} established the existence and uniqueness of global weak solutions and also proved the existence of smooth solutions for the three-dimensional LLS equation with small and regular initial data. However, there are a few relevant results about the global existence of smooth solutions for the LLS equation in other suitable spaces. Inspired by \cite{Lin}, we notice that some conclusions in the Sobolev space still hold in the Morrey space, such as they both have the interpolation inequality, similar linear estimates of the dissipative Schr\"{o}dinger semigroup, see \cite[Lemma 3.1]{Melcher1} and \cite{Lin}. Based on the fact that the regularity theory has been successfully used in the LLG equation \cite{Lin, Melcher1}, we try to apply the modified approach to improve the regularity of weak solutions for the LLS equation. Hence, our goal is to investigate the existence and uniqueness of the global smooth solutions for the LLS equation with the lower regular initial data in the Morrey space. In the context of proving classical solutions persist as long as spatial gradients remain bounded, we need to get the boundedness of $\|\nabla\boldsymbol{m}\|_{L^{\infty}(\mathbb{R}^{n})}$. It is necessary to seek for some approaches to address the nonlinear parts due to their strong nonlinearity. There is a widely used method called moving frames, which reveals originally the structure of the nonlinear parts by describing the move of coordinates on the tangent bundle. Applying moving frames, we transform the LLS equation into the covariant complex Ginzburg-Landau equation. In view of the LLS equation has more terms about $\boldsymbol{v}$ than the LLG equation, hence some new estimates with respect to $\boldsymbol{v}(t)$ are introduced by nonlinear estimates. It is worth mentioning that we give a detailed proof for choosing the Coulomb gauge.

Moreover, unlike the LLG equation has invariance, our nonlinear parts have an extra term $(\boldsymbol{v}\cdot\nabla)\boldsymbol{m}$, which can generate the invariance of the LLS equation out of work. Specifically, when we prove a particular energy inequality in the regularity theory, it is desired that the estimate of the $M^{2,3}$-norm of $\partial_{t}\boldsymbol{m}$ can be controlled by the $M^{2,3}$-norm of $\nabla\boldsymbol{m}$ in appropriate cylinders $P_{r}(z)$ with the same coefficients about $r$. However, the exponents of $r$ are different on account of the invariance failure. Therefore, instead of estimating directly in the cylinder, we first consider it in the whole space $\mathbb{R}^{n}$, and then by means of selecting a suitable test function with compact support in the cylinder, the equation is still defined in the whole space yields that we can use the Fourier transform. Finally, we combine the Plancherel equality with the interpolation inequality to achieve the purpose of balancing the coefficients.

Now, we state our main result in the following.
\begin{theorem}\label{main-th}
For any given $\boldsymbol{m}_{\infty}\in\mathbb{S}^{2}$ and the initial data $\boldsymbol{m}_{0}:\mathbb{R}^{3}\rightarrow\mathbb{S}^{2}$ with
\begin{align*}
\boldsymbol{m}_{0}-\boldsymbol{m}_{\infty}\in M^{2,3,1}(\mathbb{R}^{3})\cap M^{3,3,1}(\mathbb{R}^{3}),
\end{align*}
if there exist a sufficiently small constant $\rho>0$ such that
\begin{align*}
\|\nabla\boldsymbol{m}_{0}\|_{M^{3,3}(\mathbb{R}^{3})}+\sup\limits_{t>0}\|\boldsymbol{v}(t)\|_{M^{3,3}(\mathbb{R}^{3})}
+\sup\limits_{t>0}\sqrt{t}\|\boldsymbol{v}(t)\|_{L^{\infty}(\mathbb{R}^{3})}<\rho,
\end{align*}
then for any integer $k>3$ and all $t>0$, the equation \eqref{main-eq} have a global smooth solution satisfying
\begin{align*}
\boldsymbol{m}-\boldsymbol{m}_{\infty}\in C^{0}([0,\infty);M^{2,3,k}(\mathbb{R}^{3};\mathbb{R}^{3}))\cap C^{0}((0,\infty);M^{2,3,\infty}(\mathbb{R}^{3};\mathbb{R}^{3})).
\end{align*}
Moreover, $\boldsymbol{m}$ is unique and stable.
\end{theorem}

Notice that there are two main tasks we need to address, one is to extend local smooth solutions and the other is to improve the regularity of weak solutions. We first establish the higher order energy estimate in Lemma \ref{higher-energy-es}, which is necessary for getting the local existence of the classical solution with the initial data $\boldsymbol{m}_{0}\in M^{2,3,\sigma}_{\ast}(\mathbb{R}^{3})(k\geq3)$ (see Proposition \ref{global-smooth-solu}). Observing that the boundedness of $\|\nabla\boldsymbol{m}\|_{L^{\infty}(\mathbb{R}^{3})}$ is vital for extending the local solutions to global solutions. In order to deal with the complicated nonlinear parts, we establish the covariant Landau-Lifshitz-Slonczewski equation \eqref{covariant-main-eq} in Proposition \ref{pro2.1} by using moving frames. Subsequently, we estimate the linear and nonlinear estimates in Lemma \ref{Lemma3.1}--Lemma \ref{Lemma3.4} for getting the boundedness of $\|\nabla \boldsymbol{m}\|_{L^{\infty}(\mathbb{R}^{3})}$ (see Lemma \ref{Lemma3.6}). So far we solve the first problem, and then consider the second problem. When the initial data $\boldsymbol{m}_{0}$ belongs to $ M^{2,3,1}_{\ast}(\mathbb{R}^{3})$, namely the regularity of initial data are insufficiently well, we first utilize an approximation result to get the existence of global weak solutions (see subsection \ref{Sec6.1}). Furthermore, we apply the regularity theory to improve the regularity of weak solutions. At this point, our goal is proving $\boldsymbol{m}\in C^{0,\alpha}(P_{r}(z))$ by the relationship between the H\"{o}lder space and the BMO space (see Lemma \ref{Lpa}) and verifying a small energy condition by Lemma \ref{E(z,r)}. Finally, we complete the proof of the main theorem by using Theorem \ref{holder space}.

This paper is arranged as follows. In Section \ref{Sec2}, we introduce our main function space and recall some useful lemmas.
In Section \ref{Sec3}, we show the higher order energy estimates and get the uniqueness of solutions. In Section \ref{Sec4}, we investigate the covariant Landau-Lifshitz-Slonczewski equation \eqref{covariant-main-eq} in Proposition \ref{pro2.1} by using moving frames. In Section \ref{Sec5}, we establish the linear and nonlinear estimates to get the boundedness of $\|\nabla \boldsymbol{m}\|_{L^{\infty}(\mathbb{R}^{3})}$. We show some important conclusions and then devote to giving the proof of our main results (see Proposition \ref{holder space}) in the last section.

\section{Definition and notation}\label{Sec2}
In this section, we recall some definitions and useful lemmas.
\begin{definition}[Morrey space]
For any $1\leq p<+\infty,~0\leq q\leq n$, the Morrey space $M^{p,q}(\mathbb{R}^{n})$ is defined by
\begin{align*}
M^{p,q}(\mathbb{R}^{n}):=\left\{f\in L_{\mathrm{loc}}^{p}(\mathbb{R}^{n});
\|f\|_{M^{p,q}(\mathbb{R}^{n})}^{p}
=\sup _{x \in \mathbb{R}^{n},0<r<+\infty}r^{q-n}\int_{B_{r}(x)}|f(y)|^{p}dy<+\infty\right\}.
\end{align*}
\end{definition}
Similar to the Sobolev space, we can also define the weak derivative in the Morrey space, and then by the fact that $\boldsymbol{m}\times\Delta\boldsymbol{m}=\nabla\cdot(\boldsymbol{m}\times\nabla\boldsymbol{m})$, we give the following definition.
\begin{definition}\label{def-weak-solu}
For a given $\boldsymbol{v}$ and $T<+\infty$, assume that
\begin{align*}
\boldsymbol{m}\in L^{\infty}\left((0, T);M^{2,3,1}_{\ast}(\mathbb{R}^{3};\mathbb{S}^{2})\right),\quad
\frac{\partial\boldsymbol{m}}{\partial t}\in M^{2,3}\left((0,T);M^{2,3}(\mathbb{R}^{3};\mathbb{R}^{3})\right).
\end{align*}
If $\boldsymbol{m}$ satisfies \eqref{main-eq} in the sense of distribution, that is,
\begin{align*}
&\left\langle\frac{\partial\boldsymbol{m}}{\partial t},\Phi\right\rangle_{M^{2,3}(\mathbb{R}^{3})}
+\langle(\boldsymbol{v}\cdot\nabla)\boldsymbol{m},\Phi\rangle_{M^{2,3}(\mathbb{R}^{3})}
+\langle\boldsymbol{m}\times(\boldsymbol{v}\cdot\nabla)\boldsymbol{m},\Phi\rangle_{M^{2,3}(\mathbb{R}^{3})}
+\varepsilon\langle\nabla \boldsymbol{m},\nabla \Phi\rangle_{M^{2,3}(\mathbb{R}^{3})}\\
&=\langle\boldsymbol{m} \times \nabla\boldsymbol{m}, \nabla \Phi\rangle_{M^{2,3}(\mathbb{R}^{3})}
+\varepsilon\langle|\nabla \boldsymbol{m}|^{2}\boldsymbol{m}, \Phi\rangle_{M^{2,3}(\mathbb{R}^{3})},
\end{align*}
for any $t\in(0,T)$ and for any  $\Phi\in M^{2,3,1}(\mathbb{R}^{3})\cap L^{\infty}(\mathbb{R}^{3};\mathbb{R}^{3})$, we say that $\boldsymbol{m}$ is a weak solution of \eqref{main-eq}.
\end{definition}

\begin{definition}
For any $1\leq p<+\infty,~0\leq q\leq n$,
\begin{align*}
M^{p,q,k}\left(\mathbb{R}^{n}\right):=\left\{\nabla^k f\in M^{p,q}\left(\mathbb{R}^{n}\right)~|~k\geq 1\right\}.
\end{align*}
\end{definition}
Besides, for any $1\leq p<+\infty,~0\leq q\leq n$ and $k\geq 1$, we write:
\begin{align*}
M^{p,q,k}_{\ast}\left(\mathbb{R}^{3};\mathbb{S}^{2}\right):=\left\{\boldsymbol{m}: \mathbb{R}^{3} \rightarrow \mathbb{S}^{2}~|~
\nabla^k({\boldsymbol{m}-\boldsymbol{m}_{\infty}})\in M^{p,q}\left(\mathbb{R}^{3};\mathbb{R}^{3}\right)\right\}.
\end{align*}

Let $d(x,y):\mathbb{R}^{n} \times \mathbb{R}^{n}\rightarrow \mathbb{R}_{+}$ represents the distance in $\mathbb{R}^{n}$ and $\tilde{n}$ denotes the Hausdorff measure with respect to $d$. We define the Riesz potential as follows.
\begin{definition}
For any $f(x)\in M^{p,n}(\mathbb{R}^{n}),\;1\leq p\leq+\infty$ and $0<\alpha\leq\tilde{n}$, then the $\alpha$-order Riesz potential is defined by
\begin{align*}
I_{\alpha}(x)=\int_{\mathbb{R}^{n}}\frac{|f(y)|}{d(x,y)^{\tilde{n}-\alpha}}dy,\quad x\in \mathbb{R}^{n},
\end{align*}
where $d(x,y)=|x-y|$.
\end{definition}

\begin{lemma}[\cite{Lin}]\label{Lemma1.1}
For any $0<\alpha\leq \tilde{n},~0<q\leq n,~1<p<\frac{q}{\alpha}$, if $f(x)\in M^{p,n}(\mathbb{R}^{n})\cap M^{p,q}(\mathbb{R}^{n})$,
then $I_\alpha(f)\in M^{\tilde{p},n}(\mathbb{R}^{n})\cap M^{\tilde{p},q}(\mathbb{R}^{n})$ and satisfies
\begin{align*}
\|I_\alpha(f)\|_{M^{\tilde{p},q}(\mathbb{R}^{n})}\leq \|f\|_{M^{p,q}(\mathbb{R}^{n})},
\end{align*}
where $\tilde{p}=\frac{pq}{q-p\alpha}$.
\end{lemma}

Throughout this paper, let $S(t)$ represent the semigroup generated by the operator $(\varepsilon-i)\Delta$. In the following, we give some basic estimates of $S(t)f$ in the Morrey space when $f\in M^{p.q}(\mathbb{R}^{n})$.
\begin{lemma}[\cite{Lin}]\label{Lemma1.2}
$(1)$~If $f\in M^{p,q}(\mathbb{R}^{n}),~1<p<+\infty,~0\leq q\leq n$, then $S(t)f\in M^{p,q}(\mathbb{R}^{n})$ and satisfies
\begin{align*}
\|S(t)f\|_{M^{p,q}(\mathbb{R}^{n})}&\leq C\|f\|_{M^{p,q}(\mathbb{R}^{n})},\\
\|\nabla(S(t)f)\|_{M^{p,q}(\mathbb{R}^{n})}&\leq C t^{-\frac{1}{2}}\|f\|_{M^{p,q}(\mathbb{R}^{n})}.
\end{align*}
$(2)$~If $f\in M^{p,q}(\mathbb{R}^{n}),~1<p<+\infty,~0\leq q\leq n$, then $S(t)f\in M^{\tilde{p},q}(\mathbb{R}^{n})$ and satisfies
\begin{align*}
\|S(t)f\|_{M^{\tilde{p},q}(\mathbb{R}^{n})}
&\leq C t^{-\frac{q}{2}(\frac{1}{p}-\frac{1}{\tilde{p}})}\|f\|_{M^{p,q}(\mathbb{R}^{n})},\\
\|\nabla(S(t)f)\|_{M^{\tilde{p},q}(\mathbb{R}^{n})}
&\leq Ct^{-\frac{1}{2}-\frac{q}{2}(\frac{1}{p}-\frac{1}{\tilde{p}})}\|f\|_{M^{p,q}(\mathbb{R}^{n})},
\end{align*}
where $S(t)$ denotes the semigroup, the constant $C>0$ is independent of $t$ and $\tilde{p}\in [p,~p(n+1)]$.
\end{lemma}

\section{Higher order energy estimate and Uniqueness}\label{Sec3}
In this section, we establish the higher order energy estimate which is essential for proving the local existence of classical solutions and showing the uniqueness of the solution.
Since $|\boldsymbol{m}|=1$, adding $\boldsymbol{m}\times \eqref{main-eq}$ into $\varepsilon\cdot$\eqref{main-eq} and omitting some coefficients for simplicity, one has
\begin{align*}
\frac{\partial\boldsymbol{m}}{\partial t}+\boldsymbol{m}\times \frac{\partial\boldsymbol{m}}{\partial t}
+(\boldsymbol{v}\cdot\nabla)\boldsymbol{m}+\boldsymbol{m}\times(\boldsymbol{v}\cdot\nabla)\boldsymbol{m}
=\varepsilon(\Delta\boldsymbol{m}+|\nabla\boldsymbol{m}|^{2}\boldsymbol{m}),
\end{align*}
and then we can easily deduce the following energy equality:
\begin{align*}
\varepsilon\frac{\partial}{\partial t}E(\boldsymbol{m})
+\int_{\mathbb{R}^{n}}\left|\frac{\partial \boldsymbol{m}}{\partial t}\right|^{2} dx
+\int_{\mathbb{R}^{n}}(\boldsymbol{v}\cdot\nabla)\boldsymbol{m}\frac{\partial\boldsymbol{m}}{\partial t}dx
+\int_{\mathbb{R}^{n}}\boldsymbol{m}\times(\boldsymbol{v}\cdot\nabla)\boldsymbol{m}\frac{\partial\boldsymbol{m}}{\partial t}dx
=0.
\end{align*}

\begin{lemma}\label{higher-energy-es}
Suppose that $\boldsymbol{m}\in C^{0}\left((0,T);M^{2,3,\infty}_{\ast}(\mathbb{R}^{3})\right)$ is the smooth solution of \eqref{main-eq}, for any $t\in(0,T)$ and positive integer $k\geq3$, we have
\begin{align}\label{high-en-es}
\|\nabla \boldsymbol{m}(t)\|^2_{M^{2,3,k-1}(\mathbb{R}^{3})}
\leq e^{C(t)}\|\nabla \boldsymbol{m}(0)\|^2_{M^{2,3,k-1}(\mathbb{R}^{3})},
\end{align}
where
\begin{align*}
C(t)&=\int_{0}^{t}\|\boldsymbol{v}\|^2_{L^{\infty}(\mathbb{R}^{3})}(2+\|\nabla \boldsymbol{m}\|^2_{L^{\infty}(\mathbb{R}^{3})})
+\|\nabla \boldsymbol{m}\|^2_{L^{\infty}(\mathbb{R}^{3})}ds.
\end{align*}
\end{lemma}

\begin{proof}
Multiplying \eqref{main-eq} by $\partial^k \boldsymbol{m}$ with a given integer $k$, it holds that
\begin{align}\label{high1.0}
\begin{split}
\langle\partial^k(\frac{\partial\boldsymbol{m}}{\partial t}),\partial^{k}\boldsymbol{m}\rangle
&=-\langle\partial^k(\boldsymbol{v}\cdot\nabla)\boldsymbol{m},\partial^{k}\boldsymbol{m}\rangle
-\langle\partial^k\left(\boldsymbol{m}\times(\boldsymbol{v}\cdot\nabla)\boldsymbol{m}\right),\partial^{k}\boldsymbol{m}\rangle\\
&\quad-\langle\partial^k(\boldsymbol{m}\times\Delta \boldsymbol{m}),\partial^{k}\boldsymbol{m}\rangle
+\varepsilon\langle\partial^k(\Delta\boldsymbol{m}+|\nabla\boldsymbol{m}|^{2}\boldsymbol{m}),\partial^{k}\boldsymbol{m}\rangle.
\end{split}
\end{align}
We first consider the term on the left-hand side of \eqref{high1.0}, and it is easy to get
\begin{align}\label{high1.3}
\langle\partial^k(\frac{\partial\boldsymbol{m}}{\partial t}),\partial^{k}\boldsymbol{m}\rangle
=\frac{d}{dt}\|\partial^{k}\boldsymbol{m}\|^{2}_{M^{2,3}}
=\frac{d}{dt}\|\partial^{k-1}\nabla\boldsymbol{m}\|^{2}_{M^{2,3}}.
\end{align}
Next, we estimate the first term on the right-hand side of \eqref{high1.0}. Using the integration by parts and H\"{o}lder's inequality, we have
\begin{align*}
|\langle\partial^k(\boldsymbol{v}\cdot\nabla)\boldsymbol{m},\partial^{k}\boldsymbol{m}\rangle|
&= |\langle\partial^{k-1}(\boldsymbol{v}\cdot\nabla)\boldsymbol{m},\partial^{k} \nabla\boldsymbol{m}\rangle| \\
&\leq \|(\boldsymbol{v}\cdot\nabla)\boldsymbol{m}\|_{M^{2,3,k-1}(\mathbb{R}^{3})} \cdot \|\nabla\boldsymbol{m}\|_{M^{2,3,k}(\mathbb{R}^{3})}\\
&\leq\|\boldsymbol{v}\|_{L^{\infty}(\mathbb{R}^{3})}\|\nabla\boldsymbol{m}\|_{M^{2,3,k-1}(\mathbb{R}^{3})}
\|\nabla\boldsymbol{m}\|_{M^{2,3,k}(\mathbb{R}^{3})}.
\end{align*}
Similarly, one has
\begin{align*}
|\langle\partial^k\left(\boldsymbol{m}\times(\boldsymbol{v}\cdot\nabla)\boldsymbol{m}\right),\partial^{k}\boldsymbol{m}\rangle|
\leq \|\partial^{k-1}\left(\boldsymbol{m}\times(\boldsymbol{v}\cdot\nabla)\boldsymbol{m}\right)\|_{M^{2,3}(\mathbb{R}^{3})}
\|\nabla \boldsymbol{m}\|_{M^{2,3,k}(\mathbb{R}^{3})}.
\end{align*}
Notice that
\begin{align*}
\partial^{k-1}\left(\boldsymbol{m}\times(\boldsymbol{v}\cdot\nabla)\boldsymbol{m}\right)
=\boldsymbol{m}\times \partial^{k-1}(\boldsymbol{v}\cdot\nabla)\boldsymbol{m}+Q_1.
\end{align*}
By the fact that the term $Q_1$ only contains $\partial^\alpha\boldsymbol{m}$ and $\partial^{\beta} (\boldsymbol{v}\cdot\nabla)\boldsymbol{m}$, where $1\leq\alpha\leq k-1,~0\leq\beta\leq k-2$, namely $Q_1$ has no absolute term $\boldsymbol{m}$, and then applying the Moser inequality \cite[p.9]{Taylor}, one has
\begin{align*}
\|Q_1\|_{M^{2,3}(\mathbb{R}^{3})}
&\leq \|\nabla\boldsymbol{m}\|_{L^{\infty}(\mathbb{R}^{3})} \|(\boldsymbol{v}\cdot\nabla)\boldsymbol{m}\|_{M^{2,3,k-2}(\mathbb{R}^{3})}
+\|(\boldsymbol{v}\cdot\nabla)\boldsymbol{m}\|_{L^{\infty}(\mathbb{R}^{3})} \|\nabla\boldsymbol{m}\|_{M^{2,3,k-2}(\mathbb{R}^{3})}\\
&\leq \|\nabla\boldsymbol{m}\|_{L^{\infty}(\mathbb{R}^{3})} \|\boldsymbol{v}\|_{L^{\infty}(\mathbb{R}^{3})}
\|\nabla\boldsymbol{m}\|_{M^{2,3,k-2}(\mathbb{R}^{3})} \\
&\quad+\|\boldsymbol{v}\|_{L^{\infty}(\mathbb{R}^{3})} \|\nabla\boldsymbol{m}\|_{L^{\infty}(\mathbb{R}^{3})}
\|\nabla\boldsymbol{m}\|_{M^{2,3,k-2}(\mathbb{R}^{3})} \\
&\leq \|\boldsymbol{v}\|_{L^{\infty}(\mathbb{R}^{3})}\|\nabla\boldsymbol{m}\|_{L^{\infty}(\mathbb{R}^{3})}
\|\nabla\boldsymbol{m}\|_{M^{2,3,k-1}(\mathbb{R}^{3})}.
\end{align*}
Since $|\boldsymbol{m}|=1$, we can easily obtain
\begin{align*}
\|\boldsymbol{m}\times \partial^{k-1}(\boldsymbol{v}\cdot\nabla)\boldsymbol{m}\|_{M^{2,3}(\mathbb{R}^{3})}
&\leq \|(\boldsymbol{v}\cdot\nabla)\boldsymbol{m}\|_{M^{2,3,k-1}(\mathbb{R}^{3})} \\
&\leq \|\boldsymbol{v}\|_{L^{\infty}(\mathbb{R}^{3})} \|\nabla\boldsymbol{m}\|_{M^{2,3,k-1}(\mathbb{R}^{3})}.
\end{align*}
Combining the above arguments including $\boldsymbol{v}$, by Young's inequality, we have
\begin{align}\label{high1.5}
\begin{split}
&|\langle\partial^k(\boldsymbol{v}\cdot\nabla)\boldsymbol{m},\partial^{k}\boldsymbol{m}\rangle|
+|\langle\partial^k\left(\boldsymbol{m}\times(\boldsymbol{v}\cdot\nabla)\boldsymbol{m}\right),\partial^{k}\boldsymbol{m}\rangle| \\
&\leq \|\boldsymbol{v}\|_{L^{\infty}(\mathbb{R}^{3})}\|\nabla\boldsymbol{m}\|_{M^{2,3,k-1}(\mathbb{R}^{3})}
\|\nabla\boldsymbol{m}\|_{M^{2,3,k}(\mathbb{R}^{3})}+\|\nabla\boldsymbol{m}\|_{M^{2,3,k}(\mathbb{R}^{3})}\\
&\quad\left(\|\boldsymbol{v}\|_{L^{\infty}(\mathbb{R}^{3})}\|\nabla\boldsymbol{m}\|_{L^{\infty}(\mathbb{R}^{3})}
\|\nabla\boldsymbol{m}\|_{M^{2,3,k-1}(\mathbb{R}^{3})}+\|\boldsymbol{v}\|_{L^{\infty}(\mathbb{R}^{3})} \|\nabla\boldsymbol{m}\|_{M^{2,3,k-1}(\mathbb{R}^{3})}\right)\\
&\leq \delta\|\nabla \boldsymbol{m}\|^2_{M^{2,3,k}(\mathbb{R}^{3})}+C(\delta)\|\nabla \boldsymbol{m}\|^2_{M^{2,3,k-1}(\mathbb{R}^{3})} \|\boldsymbol{v}\|^2_{L^{\infty}(\mathbb{R}^{3})}
\left(2+\|\nabla \boldsymbol{m}\|^2_{L^{\infty}(\mathbb{R}^{3})}\right).
\end{split}
\end{align}
Similarly, we estimate the remaining terms of \eqref{high1.0} and it holds that
\begin{align*}
&|\langle\partial^k(\boldsymbol{m}\times\Delta \boldsymbol{m}),\partial^{k}\boldsymbol{m}\rangle|
=|\langle\partial^k(\boldsymbol{m}\times\nabla\boldsymbol{m}),\partial^{k}\nabla\boldsymbol{m}\rangle|\\
&\leq |\langle\boldsymbol{m}\times\partial^k\nabla\boldsymbol{m},\partial^{k}\nabla\boldsymbol{m}\rangle|
+|\langle Q_2, \partial^{k}\nabla\boldsymbol{m}\rangle|\\
&\leq\|Q_2\|_{M^{2,3}(\mathbb{R}^{3})} \|\nabla\boldsymbol{m}\|_{M^{2,3,k}(\mathbb{R}^{3})},
\end{align*}
where the term $Q_2$ only includes the derivative parts and do not have absolute term $\boldsymbol{m}$. Besides, we have
\begin{align}\label{high1.1}
\begin{split}
&|\langle\partial^k(\boldsymbol{m}\times\Delta \boldsymbol{m}),\partial^{k}\boldsymbol{m}\rangle|\\
&\leq\|\nabla \boldsymbol{m}\|_{L^{\infty}(\mathbb{R}^{3})}\|\nabla\boldsymbol{m}\|_{M^{2,3,k-1}(\mathbb{R}^{3})}
\|\nabla\boldsymbol{m}\|_{M^{2,3,k}(\mathbb{R}^{3})}\\
&\leq  C(\delta)\|\nabla\boldsymbol{m}\|^{2}_{L^{\infty}(\mathbb{R}^{3})}\|\nabla \boldsymbol{m}\|^{2}_{M^{2,3,k-1}(\mathbb{R}^{3})}+\delta\|\nabla \boldsymbol{m}\|^{2}_{M^{2,3,k}(\mathbb{R}^{3})},
\end{split}
\end{align}
and
\begin{align}\label{high1.4}
\begin{split}
\langle\partial^k\Delta\boldsymbol{m},\partial^{k}\boldsymbol{m}\rangle
=-\langle\partial^k\nabla\boldsymbol{m},\partial^{k}\nabla\boldsymbol{m}\rangle
=-\|\partial^k\nabla\boldsymbol{m}\|^{2}_{M^{2,3}(\mathbb{R}^{3})},
\end{split}
\end{align}
\begin{align}\label{high1.2}
\begin{split}
|\langle\partial^k(|\nabla\boldsymbol{m}|^{2}\boldsymbol{m}),\partial^{k}\boldsymbol{m}\rangle|
\leq\|\nabla\boldsymbol{m}\|^{2}_{L^{\infty}(\mathbb{R}^{3})}\|\nabla\boldsymbol{m}\|^{2}_{M^{2,3,k-1}(\mathbb{R}^{3})}.
\end{split}
\end{align}
Summing over all $k\geq3$ and applying \eqref{high1.3}--\eqref{high1.2}, we obtain
\begin{align}\label{high1.6}
\begin{split}
&\frac{d}{dt}\|\nabla\boldsymbol{m}\|^{2}_{M^{2,3,k-1}}
+\varepsilon\|\nabla\boldsymbol{m}\|^{2}_{M^{2,3,k}(\mathbb{R}^{3})}\\
&\leq \delta\|\nabla \boldsymbol{m}\|^2_{M^{2,3,k}(\mathbb{R}^{3})}+C(\delta)\|\nabla \boldsymbol{m}\|^2_{M^{2,3,k-1}(\mathbb{R}^{3})} \|\boldsymbol{v}\|^2_{L^{\infty}(\mathbb{R}^{3})}\left(2+\|\nabla \boldsymbol{m}\|^2_{L^{\infty}(\mathbb{R}^{3})}\right)\\
&\quad+C(\delta)\|\nabla\boldsymbol{m}\|^{2}_{L^{\infty}(\mathbb{R}^{3})}\|\nabla \boldsymbol{m}\|^{2}_{M^{2,3,k-1}(\mathbb{R}^{3})}+\delta\|\nabla \boldsymbol{m}\|^{2}_{M^{2,3,k}(\mathbb{R}^{3})}\\
&\quad+\varepsilon\|\nabla\boldsymbol{m}\|^{2}_{L^{\infty}(\mathbb{R}^{3})}\|\nabla\boldsymbol{m}\|^{2}_{M^{2,3,k-1}(\mathbb{R}^{3})}\\
&\leq 2\delta\|\nabla \boldsymbol{m}\|^2_{M^{2,3,k}(\mathbb{R}^{3})}+\|\nabla \boldsymbol{m}\|^2_{M^{2,3,k-1}(\mathbb{R}^{3})}\\
&\quad\left\{C(\delta)\|\boldsymbol{v}\|^2_{L^{\infty}(\mathbb{R}^{3})}(2+\|\nabla \boldsymbol{m}\|^2_{L^{\infty}(\mathbb{R}^{3})})
+(C(\delta)+\varepsilon)\|\nabla \boldsymbol{m}\|^2_{L^{\infty}(\mathbb{R}^{3})}\right\}.
\end{split}
\end{align}
Selecting a sufficiently small constant $\delta>0$ such that the term $\|\nabla \boldsymbol{m}\|^2_{M^{2,3,k}(\mathbb{R}^{3})}$ can be absorbed by the left-hand side of \eqref{high1.6},  we can get \eqref{high-en-es} by Gronwall's inequality. Obviously, we complete the proof of Lemma \ref{higher-energy-es}.
\end{proof}

Now we have the higher order energy estimate, and then the existence of local solutions for initial data in $M^{2,3,k}(\mathbb{R}^{3})(k\geq3)$ can be shown in \cite{DingWang1,Kenig}. We notice that the boundedness of $\|\nabla \boldsymbol{m}\|_{L^{\infty}(\mathbb{R}^{3})}$ plays an important role in extending local solutions to global solutions.
\begin{proposition}\label{global-smooth-solu}
For any integer $k\geq 3$, assume that the initial data $\boldsymbol{m}_0\in M^{2,3,k}_{\ast}\left(\mathbb{R}^{3};\mathbb{S}^{2}\right)$, $\boldsymbol{v}: \mathbb{R}^{3} \times(0, \infty) \rightarrow \mathbb{R}^{3} \times \mathbb{R}^{3}$ satisfies
\begin{align*}
\boldsymbol{v} \in C^{0}\left([0, \infty);M^{2,3,k-1}_{\ast}(\mathbb{R}^{3};\mathbb{R}^{3}\times\mathbb{R}^{3})\right),
\end{align*}
then there exists a terminal time $T^{\ast}>0$, such that for any $0<T<T^{\ast}$, the classical solution of the equation \eqref{main-eq} satisfies
\begin{align*}
\boldsymbol{m}\in C^{0}\left([0, T) ;M^{2,3,k}_{\ast}(\mathbb{R}^{3};\mathbb{S}^{2})\right),\quad
\frac{\partial \boldsymbol{m}}{\partial t} \in C^{0}\left([0, T) ;M^{2,3,k-2}(\mathbb{R}^{3};\mathbb{R}^{3}\times\mathbb{R}^{3})\right).
\end{align*}
Besides, the solution $\boldsymbol{m}$ persists as long as $\|\nabla \boldsymbol{m}\|_{L^{\infty}(\mathbb{R}^{n})}<\infty$.
\end{proposition}

Next, we consider the uniqueness of weak solutions and classical solutions, which is useful in approximating local solutions. In the later section, it is also helpful for proving the Cauchy sequence.
\begin{lemma}[Uniqueness and stability]\label{Lemma1.4}
$(1)$~Assume that $\boldsymbol{m}_{1}$ and $\boldsymbol{m}_{2}$ are classical solutions of \eqref{main-eq}, for any $t\in(0,T)$, we have
\begin{align}\label{3.1}
\begin{split}
\|(\boldsymbol{m}_{1}-\boldsymbol{m}_{2})(t)\|^2_{M^{2,3}(\mathbb{R}^{3})} &\leq \exp\left(C\sum_{i=1}^2\int_{0}^{t}
\|\nabla\boldsymbol{m}_{i}\|_{L^{\infty}(\mathbb{R}^{3})}ds+\|\boldsymbol{v}\|_{L^{\infty}(\mathbb{R}^{3})}\right)\\
&\quad\|(\boldsymbol{m}_{1}-\boldsymbol{m}_{2})(0)\|^2_{M^{2,3}(\mathbb{R}^{3})},
\end{split}
\end{align}
where $C>0$ is independent of t. In particular, solutions in Proposition \ref{global-smooth-solu} are unique and stable.

$(2)$~Assume that $\boldsymbol{m}_{1}$ and $\boldsymbol{m}_{2}$ are weak solutions of \eqref{main-eq},
$\nabla\boldsymbol{m}_{i}\in L^{\infty}\left((0,T);M^{3,3}(\mathbb{R}^{3})\right)$ and for any sufficiently small constant $\eta>0,~t\in(0,T)$, it holds
\begin{align*}
\sup_{t\in(0,T)}\left(\|\nabla m_{i}(t)\|_{M^{3,3}(\mathbb{R}^{3})}+\|\boldsymbol{v}\|_{L^{\infty}(\mathbb{R}^{3})}\right)\leq \eta.
\end{align*}
Then there exists a constant $C=C(\eta,\varepsilon)$ such that
\begin{align}\label{3.2}
\|(\boldsymbol{m}_{1}-\boldsymbol{m}_{2})(t)\|^2_{M^{2,3}(\mathbb{R}^{3})}
+C\int_{0}^{t}\|\nabla(\boldsymbol{m}_{1}-\boldsymbol{m}_{2})(t)\|^2_{M^{2,3}(\mathbb{R}^{3})}ds
\leq \|(\boldsymbol{m}_{1}-\boldsymbol{m}_{2})(0)\|^2_{M^{2,3}(\mathbb{R}^{3})}.
\end{align}
In particular, solutions in Proposition \ref{global-smooth-solu} are unique and stable.
\end{lemma}

\begin{proof}
Observe that $\boldsymbol{m}_{1},~\boldsymbol{m}_{2}$ are classical solutions of \eqref{main-eq}, let $\Psi(t)=\boldsymbol{m}_{1}-\boldsymbol{m}_{2}$, and then we obtain
\begin{align*}
\Psi_{t}&=-(\boldsymbol{v}\cdot \nabla)\Psi-\Psi\times(\boldsymbol{v}\cdot \nabla)\boldsymbol{m}_{1}
-\nabla\cdot(\Psi \times \nabla\boldsymbol{m}_{1}+\boldsymbol{m}_{2} \times \nabla\Psi)\\
&\quad +\boldsymbol{m}_{2}\times(\boldsymbol{v} \cdot \nabla)\Psi+\varepsilon \left[\Delta\Psi+|\nabla\boldsymbol{m}_{1}|^{2}\Psi
+\boldsymbol{m}_{2}\nabla\Psi:\nabla(\boldsymbol{m}_{1}+\boldsymbol{m}_{2})\right],
\end{align*}
where the notation ``:'' means the matrix product. Selecting $2\Psi$ as the test function and using H\"{o}lder's inequality, integration by parts, we have{\small
\begin{align}\label{1.12}
\begin{split}
\frac{d}{dt}\|\Psi\|^2_{M^{2,3}(\mathbb{R}^{3})}+2\varepsilon\|\nabla\Psi\|^2_{M^{2,3}(\mathbb{R}^{3})}
&\leq\|\boldsymbol{v}\|_{L^{\infty}(\mathbb{R}^{3})}
\|\Psi\|_{M^{2,3}(R^3)}\|\nabla\Psi\|_{M^{2,3}(\mathbb{R}^{3})}+2\varepsilon\|\nabla\boldsymbol{m}_{1} \Psi\|^2_{M^{2,3}(\mathbb{R}^{3})} \\
&\quad+2\varepsilon\|\nabla\Psi\|^2_{M^{2,3}(\mathbb{R}^{3})}\|\nabla(\boldsymbol{m}_{1}
+\boldsymbol{m}_{2})\Psi\|_{M^{2,3}(\mathbb{R}^{3})} \\
&\quad+2\|\nabla\Psi\|_{M^{2,3}(\mathbb{R}^{3})}\|\Psi\times\nabla\boldsymbol{m}_{1}\|_{M^{2,3}(\mathbb{R}^{3})}.
\end{split}
\end{align}}
\!\!By Young's inequality, one has
\begin{align*}
&\frac{\mathrm{d}}{\mathrm{d}t}\|\Psi\|^2_{M^{2,3}(\mathbb{R}^{3})}+2\varepsilon\|\nabla\Psi\|^2_{M^{2,3}(\mathbb{R}^{3})}\\
&\leq c(\epsilon_{1})\|\boldsymbol{v}\|^2_{L^{\infty}(\mathbb{R}^{3})}\|\Psi\|^2_{M^{2,3}(\mathbb{R}^{3})}
+\epsilon_{1}\|\nabla\Psi\|^2_{M^{2,3}(\mathbb{R}^{3})}\\
&\quad+(2\varepsilon+c(\epsilon_{2}))\|\nabla\boldsymbol{m}_{1}\|^2_{L^{\infty}(\mathbb{R}^{3})}\|\Psi\|^2_{M^{2,3}(\mathbb{R}^{3})}
+\epsilon_{2}\|\nabla\Psi\|^2_{M^{2,3}(\mathbb{R}^{3})}\\
&\quad+c(\epsilon_{2})\|\Psi\|^2_{M^{2,3}(\mathbb{R}^{3})}\|\nabla\boldsymbol{m}_{2}\|^2_{L^{\infty}(\mathbb{R}^{3})}
+\epsilon_{3}\|\nabla\Psi\|^2_{M^{2,3}(\mathbb{R}^{3})}\\
&\quad+c(\epsilon_{3})\|\Psi\times\nabla\boldsymbol{m}_{1}\|^2_{M^{2,3}(\mathbb{R}^{3})}.
\end{align*}
Choosing suitable constants $\epsilon_{1},\epsilon_{2},\epsilon_{3}$ such that $\epsilon_{1}+\epsilon_{2}+\epsilon_{3}=2\varepsilon$, and then $\|\nabla\Psi\|^2_{M^{2,3}(R^3)}$ can be absorbed by the term on the left-hand side, that is,
\begin{align*}
&\frac{d}{dt}\|\Psi\|^2_{M^{2,3}(\mathbb{R}^{3})}\\
&\leq c(\varepsilon) \left(\|\boldsymbol{v}\|^2_{L^{\infty}(\mathbb{R}^{3})}\|\Psi\|^2_{M^{2,3}(\mathbb{R}^{3})}
+\|\nabla\boldsymbol{m}_{1}\|^2_{L^{\infty}(\mathbb{R}^{3})}\|\Psi\|^2_{M^{2,3}(\mathbb{R}^{3})}
+\|\nabla\boldsymbol{m}_{2}\|^2_{L^{\infty}(\mathbb{R}^{3})}\|\Psi\|^2_{M^{2,3}(\mathbb{R}^{3})}\right)\\
&\leq c(\varepsilon)\|\Psi\|^2_{M^{2,3}(\mathbb{R}^{3})}\left(\|\nabla\boldsymbol{m}_{1}\|^2_{L^{\infty}(\mathbb{R}^{3})}
+\|\nabla\boldsymbol{m}_{2}\|^2_{L^{\infty}(\mathbb{R}^{3})}+\|\boldsymbol{v}\|^2_{L^{\infty}(\mathbb{R}^{3})}\right).
\end{align*}
Applying Gronwall's inequality, we find that
\begin{align*}
\|\Psi(t)\|^2_{M^{2,3}(\mathbb{R}^{3})} \leq e^{C(t)}\|\Psi(0)\|^2_{M^{2,3}(\mathbb{R}^{3})},
\end{align*}
where $C(t)=c(\varepsilon)\int_{0}^{t}\|\nabla\boldsymbol{m}_{1}\|^2_{L^{\infty}(\mathbb{R}^{3})}
+\|\nabla\boldsymbol{m}_{2}\|^2_{L^{\infty}(\mathbb{R}^{3})}+\|\boldsymbol{v}\|^2_{L^{\infty}(\mathbb{R}^{3})}ds$.
Then, we get \eqref{3.1}.

Next, we consider the weak solution. Inserting the test function into \eqref{main-eq} and applying H\"{o}lder's inequality, we get
\begin{align}\label{1.13}
\begin{split}
\frac{\mathrm{d}}{\mathrm{d}t}\|\Psi\|^2_{M^{2,3}(\mathbb{R}^{3})}+2\varepsilon\|\nabla\Psi\|^2_{M^{2,3}(\mathbb{R}^{3})}
&\leq \|\boldsymbol{v}\|_{L^{\infty}(\mathbb{R}^{3})}\|\Psi\|_{M^{2,3}(\mathbb{R}^{3})}\|\nabla\Psi\|_{M^{2,3}(\mathbb{R}^{3})} \\
&\quad+\|\nabla\boldsymbol{m}_{1}\|^2_{M^{3,3}(\mathbb{R}^{3})}\|\Psi\|^2_{M^{6,3}(\mathbb{R}^{3})} \\
&\quad+\|\nabla\boldsymbol{m}_{1}\|_{M^{3,3}(\mathbb{R}^{3})}\|\Psi\|_{M^{6,3}(\mathbb{R}^{3})}
\|\nabla\Psi\|_{M^{2,3}(\mathbb{R}^{3})} \\
&\quad+\|\nabla\boldsymbol{m}_{2}\|_{M^{3,3}(\mathbb{R}^{3})}\|\Psi\|_{M^{6,3}(\mathbb{R}^{3})}
\|\nabla\Psi\|_{M^{2,3}(\mathbb{R}^{3})}.
\end{split}
\end{align}
By the Gagliardo-Nirenberg-Sobolev inequality, we have
\begin{align*}
\|\Psi\|_{M^{6,3}(\mathbb{R}^{3})}\leq \|\nabla\Psi\|_{M^{2,3}(\mathbb{R}^{3})}.
\end{align*}
Substituting this into \eqref{1.13}, one has
\begin{align*}
\frac{d}{dt}\|\Psi\|^2_{M^{2,3}(\mathbb{R}^{3})}+2\varepsilon\|\nabla\Psi\|^2_{M^{2,3}(\mathbb{R}^{3})}
&\leq \|\boldsymbol{v}\|_{L^{\infty}(\mathbb{R}^{3})}\|\nabla\Psi\|^2_{M^{2,3}(\mathbb{R}^{3})}
+\|\nabla\boldsymbol{m}_{1}\|^2_{M^{3,3}(\mathbb{R}^{3})}\|\nabla\Psi\|^2_{M^{2,3}(\mathbb{R}^{3})}\\
&\quad+\|\nabla\boldsymbol{m}_{2}\|^2_{M^{3,3}(\mathbb{R}^{3})}\|\nabla\Psi\|^2_{M^{2,3}(\mathbb{R}^{3})}\\
&\leq \eta\|\nabla\Psi\|^2_{M^{2,3}(\mathbb{R}^{3})},
\end{align*}
when $\eta$ is small enough, then it holds that
\begin{align*}
\|\Psi(t)\|^2_{M^{2,3}(\mathbb{R}^{3})}+(2\varepsilon-\eta)\int_{0}^{t}\|\nabla\Psi(s)\|^2_{M^{3,3}(\mathbb{R}^{3})}ds
\leq \|\Psi(0)\|^2_{M^{2,3}(\mathbb{R}^{3})}.
\end{align*}
Obviously, we obtain \eqref{3.2}.
\end{proof}

\section{The covariant Landau-Lifshitz-Slonczewski equation}\label{Sec4}
It is difficult to estimate nonlinear parts of the LLS equation directly in view of the curl operator. Therefore, in this section, we introduce a shape operator to describe the properties of nonlinear parts by using the translation and rotation of the coordinate axis. In other words, we replace nonlinear terms with moving frames to establish the covariant Landau-Lifshitz-Slonczewski equation. Finally, we verify the invariance of the covariant equation under the Coulomb gauge.
\subsection{Moving frames}
We start with a brief introduction of the method of moving frames that leads to the covariant complex Ginzburg-Landau equation. Suppose that
\begin{align*}
\boldsymbol{m}\in C^0\left([0,T];M^{p,q,\infty}_{\ast}(\mathbb{R}^3;\mathbb{S}^2)\right),\quad
\frac{\partial{\boldsymbol{m}}}{\partial t}\in C^0\left([0,T];M^{p,q,\infty}(\mathbb{R}^3\right).
\end{align*}
By virtue of the topology and a regularization argument in \cite{Beje}, we construct a vector pair $\{X,Y\}$ such that their orthonormal tangent vector along ${\boldsymbol{m}}$. Let
\begin{align*}
X,Y:[0,T]\times{\mathbb R^3}\rightarrow{\mathbb S^2},\quad {\boldsymbol{m}}\times X=Y.
\end{align*}
Moreover, $|X|=1,~|Y|=1$ and $X,~Y$ have the same regularity as ${\boldsymbol{m}}$ with respect to time-space.
Next, we define the map
\begin{align*}
\boldsymbol{a}=(a_0, a_1,a_2,a_3)\in C^0([0,T];M^{p,q,\infty}(\mathbb{R}^3;\mathbb{R}^4)),
\end{align*}
where the coefficients $a_\alpha,~\alpha\in\{1,2,3\}$ are defined by
\begin{align*}
a_\alpha=\langle\partial_{\alpha}X,Y\rangle=-\langle\partial_{\alpha}Y,X\rangle.
\end{align*}
Specifically, $\langle \cdot,\cdot\rangle$ is the Euclidean product on $\mathbb{R}^3$, $a_\alpha$ indicates the rotation of the Y-axis around the X-axis in moving frames. By symmetry, $a_\alpha$ also can denote the rotation of the X-axis around the Y-axis.

The covariant derivative can be deduced by
\begin{align*}
\mathcal{D}_\alpha=\partial_\alpha+ia_\alpha,
\end{align*}
where $\partial_0=\frac{\partial}{\partial t},~\partial_\alpha=\frac{\partial}{\partial{x_\alpha}}$, $\alpha=1,2,3$ denotes the spatial direction.
We can show the derivative of $\boldsymbol{m}$ in time-space by plural, i.e.,
\begin{align*}
u_\alpha=\langle\partial_\alpha{\boldsymbol{m}},X\rangle+i\langle\partial_\alpha{\boldsymbol{m}},Y\rangle ,\quad \alpha=0,\ldots,3.
\end{align*}
Then $\partial_\alpha\boldsymbol{m}$ is represented by its projection in the X and Y directions, i.e.,
\begin{align}
\partial_\alpha\boldsymbol{m}={\rm Re}(u_\alpha)X+{\rm Im}(u_\alpha)Y.\label{2.6}
\end{align}
Similarly,
\begin{align*}
\partial_\alpha &X=-{\rm Re}(u_\alpha)\boldsymbol{m}+a_\alpha Y,\\
\partial_\alpha &Y=-{\rm Im}(u_\alpha)\boldsymbol{m}-a_\alpha X.
\end{align*}
Moreover, we observe that the zero torsion identity:
\begin{align*}
\mathcal{D}_\alpha u_\beta=\mathcal{D}_\beta u_\alpha,
\end{align*}
and the curvature identity:
\begin{align*}
\mathcal{R}_{\alpha\beta}:=[\mathcal{D}_\alpha,\mathcal{D}_\beta]=i(\partial_{\alpha}a_{\beta}-\partial_{\beta}a_{\alpha})
=i{\rm Im}(u_{\alpha}\bar{u}_{\beta}).
\end{align*}

Assume that the initial data $\boldsymbol{m}_0\in M^{p,q,\infty}(\mathbb{R}^3;\mathbb{S}^2)$ and $\boldsymbol{m}$ is the solution of \eqref{main-eq}. We can construct coordinate axis $\{X,Y\}$ that forms an orthonormal tangent vector along $\boldsymbol{m}$ by using moving frames.
For simplicity of analysis, representing $\{X,Y\}$ as the form of complex vectors, one has
\begin{align*}
Z=X+iY.
\end{align*}
By the orthogonality of $\boldsymbol{m}$ and $X,Y$, for any $\xi\in\mathbb{R}^3$, it is easy to see that
\begin{align*}
\langle Z,\boldsymbol{m}\rangle=0,\quad \langle Z,\boldsymbol{m}\times\xi\rangle=i\langle Z,\xi\rangle.
\end{align*}

\subsection{The covariant LLS equation}
In this subsection, we can obtain the covariant form of \eqref{main-eq} by using moving frames.
\begin{proposition}\label{pro2.1}
Assume that $\boldsymbol{m}$ is the solution in Proposition \ref{global-smooth-solu} with the initial data $\boldsymbol{m}_0\in M^{2,3,\infty}_{\ast}(\mathbb{R}^3;~\mathbb{S}^2)$,
and then \eqref{main-eq} can be transformed into the covariant LLS equation by moving frames, one has
\begin{align}\label{covariant-main-eq}
\begin{cases}
u_{0}+\sum_{k=1}^{3} v_{k}u_k=(\varepsilon-i)\sum_{k=1}^{3}\mathcal{D}_{k}u_{k},\\
\mathcal{D}_{\beta}u_{\alpha}=\mathcal{D}_{\alpha}u_{\beta},\\
\partial_{\alpha}a_{\beta}-\partial_{\beta}a_{\alpha}={\rm Im}(u_{\alpha} \bar{u}_{\beta}).
\end{cases}
\end{align}
with the notations
\begin{align*}
u_\alpha\in C^0([0,T]\times\mathbb{R}^3;\mathbb{C})\quad a_\alpha\in C^0([0,T]\times\mathbb{R}^3;\mathbb{R}),\quad \alpha=1,2,3.
\end{align*}
\end{proposition}

\begin{proof}
In the following, we mainly prove that $u_0+\sum_{k=1}^{3}v_ku_k=(\varepsilon-i)\sum_{k=1}^{3}\mathcal{D}_ku_k$, the proof of zero torsion identity and the curvature identity refer to \cite{Melcher1}. With the fact that $|\boldsymbol{m}|,~|X|,~|Y|=1$, and transforming $\Delta\boldsymbol{m},~|\nabla\boldsymbol{m}|^2$ into the covariant forms respectively by \eqref{2.6}, one has
\begin{align}\label{2.7}
\begin{split}
\Delta\boldsymbol{m}&=\sum_{k=1}^{3}\partial_k\left[{\rm Re}(u_k)X+{\rm Im}(u_k)Y\right] \\
&=\sum_{k=1}^3\left[\partial_k({\rm Re}(u_k))X+\partial_k({\rm Im}(u_k))Y-{\rm Re}^2(u_k)\boldsymbol{m}\right] \\
&\quad-\sum_{k=1}^3\left[{\rm Im}^2(u_k)\boldsymbol{m}+a_k{\rm Re}(u_k)Y-a_k{\rm Im}(u_k)X\right],
\end{split}
\end{align}
\begin{align}\label{2.8}
\begin{split}
|\nabla\boldsymbol{m}|^2=\sum_{k=1}^3\left[{\rm Re}(u_k)X+{\rm Im}(u_k)Y\right]^2
=\sum_{k=1}^3\left[{\rm Re}^2(u_k)+{\rm Im}^2(u_k)\right].
\end{split}
\end{align}
Combining \eqref{2.7} and \eqref{2.8}, we deal with the right-hand side of \eqref{main-eq} as
\begin{align*}
\boldsymbol{m}\times\Delta\boldsymbol{m}&=\sum_{k=1}^3\left[\partial_k({\rm Re}(u_k))Y-a_k{\rm Im}(u_k)Y-\partial_k({\rm Im}(u_k))X
-a_k{\rm Re}(u_k)X\right],\\
\Delta\boldsymbol{m}+|\nabla\boldsymbol{m}|^{2}\boldsymbol{m}&=\sum_{k=1}^3\left[\partial_k({\rm Re}(u_k))-a_k{\rm Im}(u_k)\right]X
+\sum_{k=1}^3\left[\partial_k({\rm Im}(u_k))-a_k{\rm Re}(u_k)\right]Y.
\end{align*}
Multiplying \eqref{main-eq} by $(X+iY)$  and then simplifying by the definition of $\mathcal{D}_k$, using the orthogonality of $X,Y,\boldsymbol{m}$, we get
\begin{align*}
(X+iY)\left[-\boldsymbol{m}\times\Delta\boldsymbol{m}+\varepsilon(\Delta\boldsymbol{m}+|\nabla\boldsymbol{m}|^{2}\boldsymbol{m})\right]
=(\varepsilon-i)\sum_{k=1}^{n}\mathcal{D}_ku_k.
\end{align*}
Considering the terms on the left-hand side of \eqref{main-eq}, we write $v=(1+i)\boldsymbol{v}$ to find that
\begin{align*}
(X+iY)\left(\frac{\partial\boldsymbol{m}}{\partial t}+(\boldsymbol{v}\cdot\nabla)\boldsymbol{m}\right)=u_0+\sum_{k=1}^{3}v_ku_k.
\end{align*}
Hence, we obtain the desired results.
\end{proof}

\begin{corollary}\label{cor1}
Along with the same assumptions as Proposition \ref{pro2.1}, then $u=(u_1,u_{2},u_3)$ solves the covariant complex Ginzburg-Landau equation:
\begin{align}\label{2.10}
\mathcal{D}_{0} u_l+\mathcal{D}_l\sum_{k=1}^{n}(v_{k}u_{k})=(\alpha-i)\sum_{k=1}^{3}(\mathcal{D}_{k}\mathcal{D}_{k} u_l
+\mathcal{R}_{l k} u_{k}),\quad l \in\{1,2,3\}.
\end{align}
with the initial data
\begin{align*}
u(0)=\langle\nabla\boldsymbol{m}_{0},X\rangle+i\langle\nabla \boldsymbol{m}_{0},Y\rangle,\quad
u(0)\in M^{p,q,\infty}(\mathbb{R}^3;\mathbb{C}^3).
\end{align*}
\end{corollary}

\begin{proof}
By the conclusion of \eqref{covariant-main-eq}, one can easily obtain that
\begin{align*}
\mathcal{D}_{0}u_l=\mathcal{D}_lu_{0}
=\mathcal{D}_l\left[(\varepsilon-i)\sum_{k=1}^{3}\mathcal{D}_{k}u_k-\sum_{k=1}^{3}v_ku_k\right]
=(\varepsilon-i)\sum_{k=1}^{3}\mathcal{D}_l(\mathcal{D}_{k}u_k),
\end{align*}
which further yields that
\begin{align*}
\mathcal{D}_{0}u_l+\mathcal{D}_l\sum_{k=1}^{3}(v_{k} u_{k})=(\varepsilon-i)\sum_{k=1}^{3}(\mathcal{D}_{k} \mathcal{D}_{k}u_l
+\mathcal{R}_{l k} u_{k}),
\end{align*}
where $l\in\{1,2,3\}$.
\end{proof}

\subsection{The Coulomb gauge}
Notice that the covariant LLS equation \eqref{covariant-main-eq} is invariant with respect to the individual choice of the moving orthonormal frame $\{X,Y\}$.
The local rotation of $\{X,Y\}$ is expressed by a rotation angle $\theta:\mathbb{R}^3\times[0,T]\rightarrow\mathbb{R}$ from the gauge transformation \cite[Proposition 2.3]{Beje}. The common gauge includes the Coulomb gauge and the Lorentz gauge, here we choose the former in terms of the feature of equations.
\begin{lemma}\label{Lemma2.2}
There exists a suitable Coulomb gauge \eqref{2.13} such that equation \eqref{covariant-main-eq} satisfies the invariance.
\begin{eqnarray}\label{2.13}
\left\{\begin{array}{ll}
\!\!\!u_{\alpha}\mapsto e^{-i\theta} u_{\alpha},\\
\!\!\!a_{\alpha}\mapsto a_{\alpha}+\partial_{\alpha}\theta,
\end{array}\right.
\end{eqnarray}
where $\theta:\mathbb{R}^3\times[0,T]\rightarrow\mathbb{R}$.
\end{lemma}

\begin{proof}
 We need to choose a suitable gauge and then verify the invariance of the Coulomb gauge. Firstly, construct a new coordinate system $\{X',Y'\}$, that is,
\begin{eqnarray}\label{2.12}
\left\{\begin{array}{ll}
\!\!\!X'=\cos{\theta}X+\sin{\theta}Y,\\
\!\!\!Y'=-\sin{\theta}X+\cos{\theta}Y.
\end{array}\right.
\end{eqnarray}
Assume that there exists $a'_{\alpha},~u'_{\alpha}$  satisfying equation \eqref{covariant-main-eq}, and we first suppose that
\begin{align*}
a'_{\alpha}=\langle\partial_{\alpha}X',Y'\rangle.
\end{align*}
Substituting \eqref{2.12} into the above equality implies that
\begin{align*}
a'_{\alpha}=\langle \partial_{\alpha}(\cos{\theta}X+\sin{\theta}Y),
-\sin{\theta}X+\cos{\theta}Y\rangle
=a_{\alpha}+\partial_{\alpha}\theta.
\end{align*}
Next, we suppose that $a'_{\alpha}$ and $u'_{\alpha}$ satisfy \eqref{covariant-main-eq}, i.e.,
\begin{align*}
\mathcal{D}'_\alpha u'_\beta=\mathcal{D}'_\beta u'_\alpha,
\end{align*}
where $\mathcal{D}'_\alpha=\partial_\alpha+ia'_\alpha$, then we have
\begin{align*}
(\partial_{\alpha}+ia'_{\alpha})u'_{\beta}=(\partial_{\beta}+ia'_{\beta})u'_{\alpha}.
\end{align*}
It is easy to get
\begin{align}
\mathcal{D}_\alpha u_\beta+i\partial_{\alpha}\theta u'_{\beta}=\mathcal{D}_\beta u_\alpha+i\partial_{\beta}\theta u'_{\alpha}.\label{2.14}
\end{align}
We want to prove that \eqref{2.14} holds for any $\alpha,\beta$ in view of the invariance. Therefore, we may further assume that $u'_{\alpha}=s(\theta)u(\alpha)$, where $s(\theta)$ is the function related to $\theta$. Recalculating \eqref{2.14}, one has
\begin{align*}
[s'(\theta)+is(\theta)][\partial_{\beta}\theta u_{\alpha}-\partial_{\alpha}\theta u_{\beta}]=0,
\end{align*}
which implies that $s'(\theta)+is(\theta)=0$. By the separate variable method, we have
\begin{align*}
s(\theta)=e^{-i\theta}u.
\end{align*}
Combining the above arguments, we can choose the Coulomb gauge \eqref{2.13} such that \eqref{covariant-main-eq} still preserve invariance.
\end{proof}

For simplicity, we write $a_{\alpha},~u_{\alpha}$ instead of $a'_{\alpha},u'_{\alpha}$. Observe that the Coulomb gauge has the following characteristics:
\begin{align*}
{\rm div}{a^{\ast}}=\sum_{k=1}^3\partial_{k}a^{\ast}_{l}=0.
\end{align*}
Substituting $a(t)$ into the above equality, we can obtain an elliptic equation:
\begin{align}
-\Delta\theta(t)={\rm div}a(t),\label{ellip-equa}
\end{align}
which implies that $\theta$ solves equation \eqref{ellip-equa} when $a(t)\in C^{0}([0,T];M^{2,3,\infty}(\mathbb{R}^3))$, where
\begin{align*}
\nabla \theta\in C^{0}([0,T];M^{2,3,\infty}(\mathbb{R}^3)).
\end{align*}

\subsection{The energy estimate of  a(t)}
Recalling equation \eqref{covariant-main-eq}, for any $\alpha, \beta=1,2,3$, we have
\begin{align*}
\partial_\alpha a_\beta-\partial_\beta a_\alpha={\rm Im}(u_{\alpha}\bar{u}_{\beta}).
\end{align*}
Under the Coulomb gauge, it holds that
\begin{align}
-\Delta a_{\alpha}=\sum_{k=1}^3 \partial_k{\rm Im}(u_{\alpha} \bar{u}_k).\label{a-u}
\end{align}
Taking into account the estimate of $u(t)$ in later sections, it is necessary to obtain the corresponding estimates for $a(t)$ in view of \eqref{a-u}.
\begin{lemma}\label{Lemma2.1}
Assume that $a(t)\in M^{p,3}(\mathbb{R}^{3}),~\nabla u(t)\in M^{3,3}(\mathbb{R}^{3})$, for any $2<p<6,~0\leq t\leq T$, then we obtain
\begin{align}\label{2.1}
\|a(t)\|_{M^{\frac{3p}{6-p},3}(\mathbb{R}^{3})}\leq\|u(t)\|^{2}_{M^{p,3}(\mathbb{R}^{3})}.
\end{align}
Moreover, for any $3<p<6$, there exists the decomposition of the term $a(t)$ such that
\begin{align}\label{2.2}
\begin{split}
\|a_{0}^{(1)}(t)\|_{M^{p,3}(\mathbb{R}^{3})}
&\leq\|u(t)\|_{M^{p,3}(\mathbb{R}^{3})}\|\nabla u(t)\|_{M^{3,3}(\mathbb{R}^{3})},\\
\|a_{0}^{(2)}(t)\|_{M^{\frac{3p}{12-2p},3}(\mathbb{R}^{3})}
&\leq\|u(t)\|^{4}_{M^{p,3}(\mathbb{R}^{3})}+\|u(t)\|^{2}_{M^{p,3}(\mathbb{R}^{3})}\|v(t)\|_{M^{\frac{3p}{6-p},3}(\mathbb{R}^{3})},
\end{split}
\end{align}
where $a_{0}(t)=a_{0}^{(1)}(t)+a_{0}^{(2)}(t)$.
\end{lemma}

\begin{proof}
We can easily get that
\begin{align*}
{\rm div}({\rm Im}(u_{\alpha}\bar{u}))=\sum_{k=1}^{3}\partial_{k}{\rm Im}(u_{\alpha} \bar{u}_{\beta})=-\Delta a_{\alpha}(t).
\end{align*}
By the definition of the Riesz potential, we have
\begin{align*}
a_{\alpha}(t)=(-\Delta)^{-\frac{1}{2}}\sum_{k=1}^3\mathcal{R}_{k}{\rm Im}(u_{\alpha}\bar{u}_{k}),
\end{align*}
where $\mathcal{R}_{k}$ are the Riesz transforms and $(-\Delta)^{-\frac{1}{2}}$ is the Riesz potential corresponding to the Fourier multiplier $\frac{1}{|\xi|}$. Obviously, one has
\begin{align*}
a(t)\approx I_{1}(|u(t)|^{2}).
\end{align*}

Using Lemma \ref{Lemma1.1}, we obtain
\begin{align*}
\|a(t)\|_{M^{\frac{3p}{6-p},3}(\mathbb{R}^{3})}\leq\||u(t)|^{2}\|_{M^{\frac{p}{2},3}(\mathbb{R}^{3})}\leq\|u(t)\|^{2}_{M^{p,3}(\mathbb{R}^{3})},
\end{align*}
then we obtain \eqref{2.1}.

Suppose that $3<p<6$, similar to the arguments of \eqref{2.1}, we write
\begin{align*}
-\Delta a_{0}^{(1)}(t)&={\rm div}(\varepsilon {\rm Im}(\bar{u}{\rm div}u)-{\rm Re}(\bar{u}{\rm div}u)),\\
-\Delta a_{0}^{(2)}(t)&={\rm div}(\varepsilon {\rm Re}(\bar{u}(a\cdot u))+{\rm Im}(\bar{u}(a\cdot u))-{\rm Im}( \bar{u}(v\cdot u))).
\end{align*}
Based on the definition of Riesz potential, we can obtain
\begin{align*}
a_{0}^{(1)}(t)\approx I_{1}(|\bar{u}||\nabla u(t)|),\quad
a_{0}^{(2)}(t)\approx I_{1}(|u(t)|^2|a(t)|+|u(t)|^2|v(t)|).
\end{align*}
Combining the two equivalent relationships with Lemma \ref{Lemma1.1}, one has
\begin{align*}
\|a_{0}^{(1)}(t)\|_{M^{p,3}(\mathbb{R}^{3})}\leq\||\bar{u}||\nabla u(t)|\|_{M^{\frac{3p}{p+3},3}(\mathbb{R}^{3})}
\leq\|u(t)\|_{M^{p,3}}\|\nabla u(t)\|_{M^{3,3}(\mathbb{R}^{3})}.
\end{align*}
Similarly, it holds
\begin{align*}
\|a_{0}^{(2)}(t)\|_{M^{\frac{3p}{12-2p},3}(\mathbb{R}^{3})}&\leq\||u|^2|a|\|_{M^{\frac{3p}{12-p},3}(\mathbb{R}^{3})}
+\||u|^2|v|\|_{M^{\frac{3p}{12-p},3}(\mathbb{R}^{3})}\\
&\leq\|u(t)\|^2_{M^{p,3}(\mathbb{R}^{3})}\|a(t)\|_{M^{\frac{3p}{6-p},3}(\mathbb{R}^{3})}+\|u(t)\|^2_{M^{p,3}(\mathbb{R}^{3})}
\|v(t)\|_{M^{\frac{3p}{6-p},3}(\mathbb{R}^{3})}\\
&\leq\|u(t)\|^4_{M^{p,3}(\mathbb{R}^{3})}+\|u(t)\|^2_{M^{p,3}(\mathbb{R}^{3})}\|v(t)\|_{M^{\frac{3p}{6-p},3}(\mathbb{R}^{3})}.
\end{align*}
Obviously, we complete the proof of \eqref{2.2}.
\end{proof}

\section{Estimates for the covariant complex Ginzburg-Landau equation}\label{Sec5}
In this section, we establish the linear estimate and the nonlinear estimate to address $\|\nabla \boldsymbol{m}\|_{L^{\infty}(\mathbb{R}^{n})}<\infty$, which is our main task to extend the local solution. Meanwhile, we utilize Duhamel's principle to solve the covariant complex Ginzburg-Landau equation. Assume that $u\in C^{0}\left([0,T];M^{2,3,\infty}(\mathbb{R}^3;\mathbb{C}^3)\right)$.

\subsection{The nonlinear estimate}
Equation \eqref{2.10} can be simplified by ${\rm div} a(t)=0$, and it holds that
\begin{align}
\frac{\partial u_l}{\partial t}=(\varepsilon-i)\Delta u_l+F_l({\bf a},u),~l=1,2,3.\label{3.1}
\end{align}
where
\begin{align*}
F_l({\bf a},u)=(\varepsilon-i)\left[\sum_{k=1}^{3}i{\rm Im} (u\bar{u_k})u_k+2i(a\cdot\nabla)u_{l}-|a|^2u_{l}\right]-ia_{0}u_{l}-(\partial_l+ia_l)(v\cdot u).
\end{align*}
Substituting $u=(u_1,u_2,u_3),~a_{0}(t)=a_{0}^{(1)}(t)+a_{0}^{(2)}(t)$ into \eqref{3.1}, we obtain $\frac{\partial u}{\partial t}
=(\varepsilon-i)\Delta u+F({\bf a},u)$. More precisely, we can easily observe that the estimate of $|a(t)|$
can be roughly viewed as the estimate of $|u(t)|^{2}$ in the Morrey space by Lemma \ref{Lemma2.1}. Similarly, the estimate of $|a_0^{(1)}(t)|$
is approximately equal to the product of $|u(t)|$ and $|\nabla u(t)|$, and the estimate of $|a_0^{(2)}(t)|$ can be considered as $|u(t)|^{4}$.
Hence, we can divide $F({\bf a},u)$ into four parts:
\begin{align*}
F({\bf a},u)=F^{(1)}+F^{(2)}+F^{(3)}+\nabla g=:f(t)-\nabla g(t),
\end{align*}
specifically,
\begin{align*}
&f=:F^{(1)}+F^{(2)}+F^{(3)},\\
&F^{(1)}:=(\varepsilon-i)\sum_{k=1}^{3}i{\rm Im}(u\bar{u_k})u_k,\\
&F^{(2)}:=2i(\varepsilon-i)(a\cdot\nabla)u-ia_0^{(1)}u,\\
&F^{(3)}:=-(\varepsilon-i)|a|^2u-ia_0^{(2)}u-ia(v\cdot u),\\
&\nabla g:=\nabla (v\cdot u).
\end{align*}

\begin{lemma}\label{Lemma3.1}
Assume that $0<p<6$, $0\leq t \leq T$, we have
\begin{align*}
&\|F^{(1)}(t)\|_{M^{\frac{p}{3},3}(\mathbb{R}^{3})}\leq \|u(t)\|^3_{M^{p,3}(\mathbb{R}^{3})},\\
&\|F^{(2)}(t)\|_{M^{\frac{p}{2},3}(\mathbb{R}^{3})}\leq \|u(t)\|^2_{M^{p,3}(\mathbb{R}^{3})}\|\nabla u(t)\|_{M^{3,3}(\mathbb{R}^{3})},\\
&\|F^{(3)}(t)\|_{M^{\frac{3p}{15-2p},3}(\mathbb{R}^{3})}\leq\|u(t)\|^5_{M^{p,3}(\mathbb{R}^{3})}+\|u(t)\|^3_{M^{p,3}(\mathbb{R}^{3})}
\|v(t)\|_{M^{\frac{3p}{6-p},3}(\mathbb{R}^{3})}.
\end{align*}
\end{lemma}

\begin{proof}
By Lemma \ref{Lemma3.1} and H\"{o}lder's inequality, one has
\begin{align*}
\|F^{(1)}(t)\|_{M^{\frac{p}{3},3}(\mathbb{R}^{3})}\leq \||u(t)|^3\|_{M^{\frac{p}{3},3}(\mathbb{R}^{3})}=\|u(t)\|^3_{M^{p,3}(\mathbb{R}^{3})}.
\end{align*}
Using the same method, we can obtain
\begin{align*}
\|F^{(2)}(t)\|_{M^{\frac{p}{2},3}(\mathbb{R}^{3})}&\leq\||a(t)|\;|\nabla
u(t)|\|_{M^{\frac{p}{2},3}(\mathbb{R}^{3})}+\||a^{(1)}_0(t)|\;|u(t)|\|_{M^{\frac{p}{2},3}(\mathbb{R}^{3})}\\
&\leq \|a(t)\|_{M^{\frac{3p}{6-p},3}(\mathbb{R}^{3})}\|\nabla u(t)\|_{M^{3,3}(\mathbb{R}^{3})}+\|a^{(1)}_0(t)\|_{M^{p,3}(\mathbb{R}^{3})}
\|u(t)\|_{M^{p,3}(\mathbb{R}^{3})}\\
&\leq \|u(t)\|^2_{M^{p,3}(\mathbb{R}^{3})}\|\nabla u(t)\|_{M^{3,3}(\mathbb{R}^{3})},\\
\|F^{(3)}(t)\|_{M^{\frac{3p}{15-2p},3}(\mathbb{R}^{3})}
&\leq\||a(t)|^2\|_{M^{\frac{3p}{12-2p},3}(\mathbb{R}^{3})}\|u(t)\|_{M^{p,3}(\mathbb{R}^{3})}
+\|a^{(2)}_0(t)\|_{M^{\frac{3p}{12-2p},3}(\mathbb{R}^{3})}\|u(t)\|_{M^{p,3}(\mathbb{R}^{3})}\\
&\quad+\|a(t)\|_{M^{\frac{3p}{6-p},3}(\mathbb{R}^{3})}\|u(t)\|_{M^{p,3}(\mathbb{R}^{3})}\|v(t)\|_{M^{\frac{3p}{6-p},3}(\mathbb{R}^{3})}\\
&\leq\|u(t)\|^5_{M^{p,3}(\mathbb{R}^{3})}+\|u(t)\|^3_{M^{p,3}(\mathbb{R}^{3})}\|v(t)\|_{M^{\frac{3p}{6-p},3}(\mathbb{R}^{3})}.
\end{align*}
Then, we get the desired estimates.
\end{proof}

\subsection{Duhamel's principle}
Assume that $u$ is the solution of \eqref{3.1} with the initial data $u_{0}$, then  equation \eqref{3.1} can be expressed by
\begin{eqnarray*}
\left\{\begin{array}{ll}
\!\!\! \frac{\partial u}{\partial t}=(\varepsilon-i)\Delta u+f-\nabla g,\\
\!\!\! u(0)=u_{0}.\\
\end{array}\right.
\end{eqnarray*}
Using the Fourier transform, we have
\begin{eqnarray}\label{sys1.1}
\left\{\begin{array}{ll}
\!\!\!\mathcal{F}(\frac{\partial u}{\partial t})=(i-\varepsilon)|\xi|^2\mathcal{F}(u)(\xi)+\mathcal{F}(f)(\xi)-\mathcal{F}(\nabla g)(\xi),\\
\!\!\!\mathcal{F}(u(0))=\mathcal{F}(u_{0})({\xi}).\\
\end{array}\right.
\end{eqnarray}
Applying Duhamel's principle to solve \eqref{sys1.1},  we obtain the expression of the solutions:
\begin{align}
\mathcal{F}(u(t))=\left[\int_0^t[\mathcal{F}(f(s))-\mathcal{F}(\nabla g(s))]e^{(\varepsilon-i)|\xi|^2s}ds
+\mathcal{F}(u_{0})\right]e^{(i-\varepsilon)|\xi|^2t}.\label{3.11}
\end{align}
Next, we define the semigroup by $S=S(t)$ which is generated by $(\varepsilon-i)\Delta$ and it is well-defined.
Calculating \eqref{3.11} by the inverse fourier transform and representing it by the semigroup, one has
\begin{align}
u(x,t)=S(t)u(0)+(S\ast f)(t)+(\nabla S\ast g)(t).\label{3.12}
\end{align}

\subsection{Semigroup estimate}
For simplify, let
\begin{align*}
K_1(t)&=\sup_{\tau\in(0,t)}\tau^{\frac{p-3}{2p}}\|u(\tau)\|_{M^{p,3}(\mathbb{R}^{3})},\\
K_2(t)&=\sup_{\tau\in(0,t)}\tau^{\frac{1}{2}}\|\nabla u(\tau)\|_{M^{3,3}(\mathbb{R}^{3})},\\
K_0(t)&=\|u(0)\|_{M^{3,3}(\mathbb{R}^{3})}.
\end{align*}
Moveover,
\begin{align*}
K(t)&=\max \left\{K_1(t),K_2(t)\right\},\\
V(t)&=\sup_{\tau\in(0,t)}\tau^{\frac{p-3}{p}}\|v(\tau)\|_{M^{\frac{3p}{6-p},3}(\mathbb{R}^{3})}.
\end{align*}
Besides, we recall a conclusion about Euler's integral. If constants $\theta,~\eta\in(0,1)$ and for any $t>0$, we know that
\begin{align*}
\int_{0}^{t}(t-s)^{-\theta}s^{-\eta}ds=ct^{1-\theta-\eta},
\end{align*}
where the constant $c=c(\theta,~\eta)$ is independently of $t$.

\begin{lemma}\label{Lemma3.2}
Assume that $3<p<5$, then we have
\begin{align}
\|u(t)\|_{M^{3,3}(\mathbb{R}^{3})}\leq K_0(t)+\left[R(t)\left(1+R(t)^2\right)\left(R(t)^2+V(t)\right)\right],\label{3.5}
\end{align}
where $t\in[0,T]$.
\end{lemma}

\begin{proof}
By the fact that the norm is $M^{3,3}(\mathbb{R}^{3})$, it holds
\begin{align*}
\|u(x,t)\|_{M^{3,3}(\mathbb{R}^{3})}&\leq\|S(t)u(0)\|_{M^{3,3}(\mathbb{R}^{3})}+\sum_{i=1}^3\|(S\ast F^{(i)})(t)\|_{M^{3,3}(\mathbb{R}^{3})} \\
&\quad+\|(\nabla S\ast g)(t)\|_{M^{3,3}(\mathbb{R}^{3})}.
\end{align*}
Using Lemma \ref{Lemma1.2}, Lemma \ref{Lemma3.1}, Euler's integral and Bochner's inequality to estimate the above five parts, we see that
\begin{align*}
\|S(t)u(0)\|_{M^{3,3}(\mathbb{R}^{3})}&\leq\|u(0)\|_{M^{3,3}(\mathbb{R}^{3})}= K_0(t),
\end{align*}
\begin{align*}
\|(S\ast F^{(1)})(t)\|_{M^{3,3}(\mathbb{R}^{3})}
&\leq\int_{0}^{t}\|S(t-s)F^{(1)}(s)\|_{M^{3,3}(\mathbb{R}^{3})}ds\\
&\leq\int_{0}^{t}(t-s)^{-\frac{9-p}{2p}}\|F^{(1)}(s)\|_{M^{\frac{p}{3},3}(\mathbb{R}^{3})}ds \\
&\leq\int_{0}^{t}(t-s)^{-\frac{9-p}{2p}}\|u(s)\|^3_{M^{p,3}(\mathbb{R}^{3})}ds\\
&\leq K_1^3(t)\int_{0}^{t}(t-s)^{-\frac{9-p}{2p}}s^{-\frac{3(p-3)}{2p}}ds  \\
&\leq K^3(t),
\end{align*}
\begin{align*}
\|(S\ast F^{(2)})(t)\|_{M^{3,3}(\mathbb{R}^{3})}
&\leq\int_{0}^{t}\|S(t-s)F^{(2)}(s)\|_{M^{3,3}(\mathbb{R}^{3})}ds\\
&\leq\int_{0}^{t}(t-s)^{-\frac{6-p}{2p}}\|F^{(2)}(s)\|_{M^{\frac{p}{2},3}(\mathbb{R}^{3})}ds \\
&\leq\int_{0}^{t}(t-s)^{-\frac{6-p}{2p}}\left[\|u(s)\|^2_{M^{p,3}(\mathbb{R}^{3})}\|\nabla u(s)\|_{M^{3,3}(\mathbb{R}^{3})}\right]ds\\
&\leq K_1(t)\cdot K_2(t)\int_{0}^{t}(t-s)^{-\frac{6-p}{2p}}s^{-\frac{1}{2}}s^{-\frac{2(p-3)}{2p}}ds  \\
&\leq K^2(t),
\end{align*}
\begin{align*}
\|(S\ast F^{(3)})(t)\|_{M^{3,3}(\mathbb{R}^{3})}
&\leq\int_{0}^{t}\|S(t-s)F^{(3)}(s)\|_{M^{3,3}(\mathbb{R}^{3})}ds\\
&\leq\int_{0}^{t}(t-s)^{-\frac{15-3p}{2p}}\|F^{(3)}(s)\|_{M^{\frac{3p}{15-2p},3}(\mathbb{R}^{3})}ds \\
&\leq\int_{0}^{t}(t-s)^{-\frac{15-3p}{2p}}
\left[\|u(s)\|^3_{M^{p,3}(\mathbb{R}^{3})}\|v(s)\|_{M^{{\frac{3p}{6-p}},3}(\mathbb{R}^{3})}
+\|u(s)\|^5_{M^{p,3}(\mathbb{R}^{3})}\right]ds \\
&\leq \left[K^5_1(t)+K^3_1(t)\cdot V(t)\right]\int_{0}^{t}(t-s)^{-\frac{5p-15}{2p}}s^{-\frac{3(p-3)}{2p}}s^{-\frac{p-3}{p}}ds \\
&\leq K^5(t)+K^3(t)\cdot V(t),
\end{align*}
\begin{align*}
\|\nabla (S\ast g)(t)\|_{M^{3,3}(\mathbb{R}^{3})}
&\leq\int_0^t\|\nabla S(t-s)g(s)\|_{M^{3,3}(\mathbb{R}^{3})}ds \\
&\leq\int_0^t(t-s)^{\frac{p-9}{2p}}\|g(s)\|_{M^{\frac{3p}{9-p}}(\mathbb{R}^{3})}ds \\
&\leq K_1(t)\cdot V(t)\int_0^t(t-s)^{-\frac{9-p}{2p}}s^{-\frac{p-3}{2p}}s^{-\frac{p-3}{p}}ds \\
&\leq K(t)\cdot V(t).
\end{align*}
Collecting all above arguments, we obtain \eqref{3.5}.
\end{proof}

\begin{lemma}\label{Lemma3.3}
Assume that $3<p<6$, then we have
\begin{align}
t^{\frac{p-3}{2p}}\| u(t)\|_{M^{p,3}(\mathbb{R}^{3})}\leq K_0(t)+\left[K(t)\left(1+K(t)^2\right)\left(K(t)^2+V(t)\right)\right],\label{3.13}
\end{align}
where $t\in[0,T]$.
\end{lemma}

\begin{proof}
We establish the estimate of solution \eqref{3.12} in the space $M^{p,3}(\mathbb{R}^{3})$ at this time, that is,
\begin{align*}
\|u(x,t)\|_{M^{p,3}(R^n)}&\leq\|S(t)u(0)\|_{M^{p,3}(\mathbb{R}^{3})}+\sum_{i=1}^3\|(S\ast F^{(i)})(t)\|_{M^{p,3}(\mathbb{R}^{3})}\\
&\quad+\|(\nabla S\ast g)(t)\|_{M^{p,3}(\mathbb{R}^{3})}.
\end{align*}
By applying Lemma \ref{Lemma1.2} and Lemma \ref{Lemma3.1}, one has
\begin{align*}
\|S(t)u(0)\|_{M^{p,3}(\mathbb{R}^{3})}&\leq t^{-\frac{p-3}{2p}}\|u(0)\|_{M^{3,3}(\mathbb{R}^{3})}= t^{-\frac{p-3}{2p}} K_0(t),
\end{align*}
\begin{align*}
\|(S\ast F^{(1)})(t)\|_{M^{p,3}(\mathbb{R}^{3})}&\leq\int_{0}^{t}(t-s)^{-\frac{3}{p}}\|F^{(1)}(s)\|_{M^{\frac{p}{3},3}(\mathbb{R}^{3})}ds\\
&\leq\int_{0}^{t}(t-s)^{-\frac{3}{p}}\|u(s)\|^3_{M^{p,3}(\mathbb{R}^{3})}ds \\
&\leq K_1^3(t)\int_{0}^{t}(t-s)^{-\frac{3}{p}}s^{-\frac{3(p-3)}{2p}}ds \\
&\leq t^{-\frac{p-3}{2p}} K^3(t),
\end{align*}
\begin{align*}
\|(S\ast F^{(2)})(t)\|_{M^{p,3}(\mathbb{R}^{3})}
&\leq\int_{0}^{t}(t-s)^{-\frac{3}{2p}}\|F^{(2)}(s)\|_{M^{\frac{p}{2},3}(\mathbb{R}^{3})}ds \\
&\leq\int_{0}^{t}(t-s)^{-\frac{3}{2p}}\left[\|u(s)\|^2_{M^{p,3}(\mathbb{R}^{3})}\|\nabla u(s)\|_{M^{3,3}(\mathbb{R}^{3})}\right]ds \\
&\leq K_1(t)\cdot K_2(t)\int_{0}^{t}(t-s)^{-\frac{3}{2p}}s^{-\frac{1}{2}}s^{-\frac{2(p-3)}{2p}}ds  \\
&\leq t^{-\frac{p-3}{2p}}K^2(t),
\end{align*}
\begin{align*}
\|(S\ast F^{(3)})(t)\|_{M^{p,3}(\mathbb{R}^{3})}&\leq\int_{0}^{t}(t-s)^{-\frac{6-p}{p}}
\|F^{(3)}(s)\|_{M^{\frac{3p}{15-2p},3}(\mathbb{R}^{3})}ds \\
&\leq\int_{0}^{t}(t-s)^{-\frac{6-p}{p}}\left[\|u(s)\|^3_{M^{p,3}(\mathbb{R}^{3})}\|v(s)\|_{M^{{\frac{3p}{6-p}},3}(\mathbb{R}^{3})}
+\|u(s)\|^5_{M^{p,3}(\mathbb{R}^{3})}\right]ds \\
&\leq [K^5_1(t)+K^3_1(t)\cdot V(t)]\int_{0}^{t}(t-s)^{-\frac{6-p}{p}}s^{-\frac{5(p-3)}{2p}}ds \\
&\leq t^{-\frac{p-3}{2p}}[K^5(t)+K^3(t)\cdot V(t)],
\end{align*}
\begin{align*}
\|\nabla (S\ast g)(t)\|_{M^{p,3}(\mathbb{R}^{3})}&\leq\int_0^t(t-s)^{-\frac{3}{p}}\|g(s)\|_{M^{\frac{3p}{9-p}}(\mathbb{R}^{3})}ds\\
&\leq K_1(t)\cdot V(t)\int_0^t(t-s)^{-\frac{3}{p}}s^{-\frac{p-3}{2p}}s^{-\frac{p-3}{p}}ds\\
&\leq t^{-\frac{p-3}{2p}}K(t)\cdot V(t).
\end{align*}
Then for any $t\in[0,T]$, we complete the proof of \eqref{3.13}.
\end{proof}

\begin{lemma}\label{Lemma3.4}
Assume that $3<p<6$, then we have
\begin{align}
t^{\frac{1}{2}}\|\nabla u(t)\|_{M^{3,3}(\mathbb{R}^{3})}\leq K_0(t)+\left[K(t)\left(1+K(t)^2\right)\left(K(t)^2+V(t)\right)\right],\label{3.19}
\end{align}
where $t\in[0,T]$.
\end{lemma}

\begin{proof}
By \eqref{3.12} and the definition of the space $M^{3,3}(\mathbb{R}^{3})$, we have
\begin{align*}
\|\nabla u(x,t)\|_{M^{3,3}(\mathbb{R}^{3})}&\leq\|\nabla S(t)u(0)\|_{M^{3,3}(\mathbb{R}^{3})}
+\sum_{i=1}^3\|\nabla(S\ast F^{(i)})(t)\|_{M^{3,3}(\mathbb{R}^{3})} \\
&+\|\nabla^2(S\ast g)(t)\|_{M^{3,3}(\mathbb{R}^{3})}.
\end{align*}
Then, we use the same method as Lemma \ref{Lemma3.2} and Lemma \ref{Lemma3.3} to get
\begin{align*}
\|\nabla S(t)u(0)\|_{M^{3,3}(\mathbb{R}^{3})}&\leq t^{-\frac{1}{2}}\|u(0)\|_{M^{3,3}(\mathbb{R}^{3})}= t^{-\frac{1}{2}} K_0(t),
\end{align*}
\begin{align*}
\|\nabla (S\ast F^{(1)})(t)\|_{M^{3,3}(\mathbb{R}^{3})}&\leq\int_{0}^{t}(t-s)^{-\frac{9}{2p}}\|F^{(1)}(s)\|_{M^{\frac{p}{3},3}(R^n)}ds \\
&\leq\int_{0}^{t}(t-s)^{-\frac{9}{2p}}\|u(s)\|^3_{M^{p,3}(\mathbb{R}^{3})}ds \\
&\leq K_1^3(t)\int_{0}^{t}(t-s)^{-\frac{9}{2p}}s^{-\frac{3p-9}{2p}}ds \\
&\leq t^{-\frac{1}{2}} K^3(t),
\end{align*}
\begin{align*}
\|\nabla (S\ast F^{(2)})(t)\|_{M^{3,3}(\mathbb{R}^{3})}
&\leq\int_{0}^{t}(t-s)^{-\frac{3}{p}}\|F^{(2)}(s)\|_{M^{\frac{p}{2},3}(\mathbb{R}^{3})}ds \\
&\leq\int_{0}^{t}(t-s)^{-\frac{3}{p}}\left[\|u(s)\|^2_{M^{p,3}(\mathbb{R}^{3})}\|\nabla u(s)\|_{M^{3,3}(\mathbb{R}^{3})}\right]ds \\
&\leq K_1(t)\cdot K_2(t)\int_{0}^{t}(t-s)^{-\frac{3}{p}}s^{-\frac{1}{2}}s^{-\frac{2(p-3)}{2p}}ds  \\
&\leq t^{-\frac{1}{2}}K^2(t),
\end{align*}
\begin{align*}
&\|\nabla (S\ast F^{(3)})(t)\|_{M^{3,3}(\mathbb{R}^{3})}\\
&\leq\int_{0}^{t}(t-s)^{-\frac{15-2p}{2p}}\|F^{(3)}(s)\|_{M^{\frac{3p}{15-2p},3}(\mathbb{R}^{3})}ds \\
&\leq\int_{0}^{t}(t-s)^{-\frac{15-2p}{2p}}
\left[\|u(s)\|^3_{M^{3,3}(\mathbb{R}^{3})}\|v(s)\|_{M^{{\frac{3p}{6-p}},3}(\mathbb{R}^{3})}+\|u(s)\|^5_{M^{p,3}(\mathbb{R}^{3})}\right]ds \\
&\leq \left[K^5_1(t)+K^3_1(t)\cdot V(t)\right]\int_{0}^{t}(t-s)^{-\frac{15-2p}{2p}}s^{-\frac{5(p-3)}{2p}}ds \\
&\leq t^{-{\frac{1}{2}}}[K^5(t)+K^3(t)\cdot V(t)],
\end{align*}
\begin{align*}
\|\nabla^2 (S\ast g)(t)\|_{M^{3,3}(\mathbb{R}^{3})}
&\leq\int_0^t(t-s)^{-\frac{9}{2p}}\|g(s)\|_{M^{\frac{3p}{9-p}}(\mathbb{R}^{3})}ds \\
&\leq K_1(t)\cdot V(t)\int_0^t(t-s)^{-\frac{9}{2p}}s^{-\frac{9}{2p}}s^{-\frac{p-3}{p}}ds \\
&\leq t^{-\frac{1}{2}}K(t)\cdot V(t).
\end{align*}
Combining the above estimates, for any $t\in[0,T]$, we get \eqref{3.19}.
\end{proof}

\begin{lemma}\label{Lemma3.5}
Assume that $3<p<6$, for any $t\in(0,T)$, we have
\begin{align*}
V(t)\leq c\max\left\{~\sup_{\tau\in(0,t)}\|v(\tau)\|_{M^{3,3}(\mathbb{R}^{3})},
~\sup_{\tau\in(0,t)}\tau^{\frac{1}{2}}\|v(\tau)\|_{L^{\infty}(\mathbb{R}^{3})}~\right\},
\end{align*}
where $c>0$ is a constant independently of $t$.
\end{lemma}

\begin{proof}
By interpolation inequalities, one has
\begin{align*}
\|v(t)\|_{M^{\frac{3p}{6-p},3}(\mathbb{R}^{3})}\leq \|v(t)\|^{\frac{6-p}{p}}_{M^{3,3}(\mathbb{R}^{3})}
 \|v(t)\|^{\frac{2p-6}{p}}_{L^{\infty}(\mathbb{R}^{3})}.
\end{align*}
Obviously, for any $t\in(0,T)$, we have
\begin{align*}
t^{\frac{p-3}{p}}\|v(t)\|_{M^{\frac{3p}{6-p},3}(\mathbb{R}^{3})}
\leq \|v(t)\|^{\frac{6-p}{p}}_{M^{3,3}(\mathbb{R}^{3})}\left[t^{\frac{1}{2}}\|v(t)\|_{L^{\infty}(\mathbb{R}^{3})}\right]^{\frac{2p-6}{p}}.
\end{align*}
For the arbitrariness of $t$, it implies that $V(t)$ satisfies the above inequality, and then applying Young's inequality, one has
\begin{align*}
V(t)\leq \|v(t)\|_{M^{3,3}(\mathbb{R}^{3})}+\;t^{\frac{1}{2}}\|v(t)\|_{L^{\infty}(\mathbb{R}^{3})}.
\end{align*}
Hence, it holds that
\begin{align*}
V(t)\leq c \max\left\{~\sup_{\tau\in(0,t)}\|v(\tau)\|_{M^{3,3}(\mathbb{R}^{3})},\;
\sup_{\tau\in(0,t)}\tau^{\frac{1}{2}}\|v(\tau)\|_{L^{\infty}(\mathbb{R}^{3})}~\right\}.
\end{align*}
We can see that $c>0$ is a constant independently of $t$.
\end{proof}

\begin{proposition}\label{pro3.1}
Assume that $3<p<5$, we have
\begin{align}
K(t)\leq K_0(t)+\left[K(t)\left(1+K(t)^2\right)\left(K(t)^2+V(t)\right)\right].\label{3.25}
\end{align}
Moreover, suppose that there exists a sufficiently small constant $\rho>0$, if
\begin{align*}
\sup_{t\in(0,T)}(K_0(t)+V(t))<\rho,
\end{align*}
then for any $t\in[0,T]$, we can obtain
\begin{align}
K(t)\leq2K_0(t).\label{3.26}
\end{align}
\end{proposition}

\begin{proof}
Recalling the definition of $K(t),~K_1(t),~K_2(t)$ and $V(t)$, and then combining Lemma \ref{Lemma3.3} with Lemma \ref{Lemma3.4}, we can obtain \eqref{3.25} directly. In the following, we argue by contradiction. Assuming that there exists $t_0\in(0,T)$ such that $K(t_0)=2K_0(t)\neq0$,
and then substituting $K(t_0)$ into \eqref{3.25}, one has
\begin{align*}
2K_0(t)\leq K_0(t)+2\left(V(t_0)+4K_0(t_0)^2\right)\left(1+4K_0(t_0)^2\right)K_0(t_0).
\end{align*}
We know that $0<K_0(t_0)+V(t_0)<\rho$ and $\rho$ is small enough, there always exists a constant $c>0$ yields that
\begin{align*}
\frac{1}{c}\leq2\left(V(t_0)+4K_0(t_0)^2\right)\left(1+4K_0(t_0)^2\right)<\frac{1}{c},
\end{align*}
which is a contradiction. Hence the continuous functions $t\mapsto K(t)$ and $t\mapsto 2K_0(t)$ do not intersect in $(0,T)$.

Now for any $t\in(0,T)$, we reassume that $K(t)>2K_0(t)$. Substituting it into \eqref{3.25}, we can get
\begin{align*}
1<\frac{1}{2}+\left(V(t)+K(t)^2\right)\left(1+K(t)^2\right),
\end{align*}
since $\lim\limits_{t \to 0} V(t)=0$ and $\lim\limits_{t \to 0} K(t)=0$, which is a contradiction. Then, we derive \eqref{3.26} from the above arguments.
\end{proof}

\begin{lemma}\label{lemma3.7}
Assume that $3<p<5,~\max\{1-\frac{9}{4p},1-\frac{15-2p}{p},1-\frac{6}{p},1-\frac{9-p}{p},\frac{3}{4}\}<\alpha<1$, then we have
\begin{align}
t^{\frac{1+\alpha}{2}}\|\nabla u(t)\|_{M^{\frac{3}{1-\alpha},3}(\mathbb{R}^{3})}
\leq K_0(t)+\left[K(t)\left(1+K(t)^2\right)\left(K(t)^2+V(t)\right)\right],\label{3.33}
\end{align}
where $t\in[0,T]$.
\end{lemma}

\begin{proof}
By \eqref{3.12} and the definition of $M^{\frac{3}{1-\alpha},3}(\mathbb{R}^{3})$, one has
\begin{align*}
\|\nabla u(x,t)\|_{M^{\frac{3}{1-\alpha},3}(\mathbb{R}^{3})}
&\leq\|\nabla S(t)u(0)\|_{M^{\frac{3}{1-\alpha},3}(\mathbb{R}^{3})}
+\sum_{i=1}^3\|\nabla(S\ast F^{(i)})(t)\|_{M^{\frac{3}{1-\alpha},3}(\mathbb{R}^{3})} \\
&\quad+\|\nabla^2(S\ast g)(t)\|_{M^{\frac{3}{1-\alpha},3}(\mathbb{R}^{3})}.
\end{align*}
Now we estimate the above five terms one by one. By Lemma \ref{Lemma1.2}, Lemma \ref{Lemma3.1}, Lemma \ref{Lemma3.6} and Bochner's inequality, we have
\begin{align*}
\|\nabla S(t)u(0)\|_{M^{\frac{3}{1-\alpha},3}(\mathbb{R}^{3})}&\leq t^{-\frac{1+\alpha}{2}}\|u(0)\|_{M^{3,3}(\mathbb{R}^{3})}
\leq t^{-\frac{1+\alpha}{2}} K_0(t),
\end{align*}
\begin{align*}
\|\nabla (S\ast F^{(1)})(t)\|_{M^{\frac{3}{1-\alpha},3}(\mathbb{R}^{3})}
&\leq\int_{0}^{t}(t-s)^{-\frac{9-\alpha p}{2p}}\|F^{(1)}(s)\|_{M^{\frac{p}{3},3}(\mathbb{R}^{3})}ds \\
&\leq\int_{0}^{t}(t-s)^{-\frac{9-\alpha p}{2p}}\|u(s)\|^3_{M^{p,3}(\mathbb{R}^{3})}ds \\
&\leq K_1^3(t)\int_{0}^{t}(t-s)^{-\frac{9-\alpha p}{2p}}s^{-\frac{3(p-3)}{2p}}ds \\
&\leq t^{-\frac{1+\alpha}{2}} K^3(t),
\end{align*}
\begin{align*}
\|\nabla (S\ast F^{(2)})(t)\|_{M^{\frac{3}{1-\alpha},3}(\mathbb{R}^{3})}
&\leq\int_{0}^{t}(t-s)^{-\frac{6+\alpha p}{2p}}\|F^{(2)}(s)\|_{M^{\frac{p}{2},3}(\mathbb{R}^{3})}ds \\
&\leq\int_{0}^{t}(t-s)^{-\frac{6+\alpha p}{2p}}\left[\|u(s)\|^2_{M^{p,3}(\mathbb{R}^{3})}
\|\nabla u(s)\|_{M^{3,3}(\mathbb{R}^{3})}\right]ds \\
&\leq K_1(t)\cdot K_2(t)\int_{0}^{t}(t-s)^{-\frac{6+\alpha p}{2p}}s^{-\frac{1}{2}}s^{-\frac{2(p-3)}{2p}}ds \\
&\leq t^{-\frac{1+\alpha}{2}}K^2(t),
\end{align*}
It suffices to show that
\begin{align*}
&\|\nabla (S\ast F^{(3)})(t)\|_{M^{\frac{3}{1-\alpha},3}(\mathbb{R}^{3})}\\
&\leq\int_{0}^{t}(t-s)^{\frac{2p-15-\alpha p}{2p}}\|F^{(3)}(s)\|_{M^{\frac{3p}{15-2p},3}(\mathbb{R}^{3})}ds \\
&\leq\int_{0}^{t}(t-s)^{\frac{2p-15-\alpha p}{2p}}
\left[\|u(s)\|^3_{M^{3,3}(\mathbb{R}^{3})}\|v(s)\|_{M^{{\frac{3p}{6-p}},3}(\mathbb{R}^{3})}
+\|u(s)\|^5_{M^{p,3}(\mathbb{R}^{3})}\right]ds \\
&\leq \left[K^5_1(t)+K^3_1(t)\cdot V(t)\right]\int_{0}^{t}(t-s)^{\frac{2p-15-\alpha p}{2p}}s^{-\frac{5(p-3)}{2p}}ds \\
&\leq t^{-\frac{1+\alpha}{2}}[K^5(t)+K^3(t)\cdot V(t)],
\end{align*}
\begin{align*}
\|\nabla^2 (S\ast g)(t)\|_{M^{\frac{3}{1-\alpha},3}(\mathbb{R}^{3})}
&\leq\int_0^t(t-s)^{-\frac{9+\alpha p}{2p}}\|g(s)\|_{M^{\frac{3p}{9-p}}(\mathbb{R}^{3})}ds \\
&\leq K_1(t)\cdot V(t)\int_0^t(t-s)^{-\frac{9+\alpha p}{2p}}s^{-\frac{9}{2p}}s^{-\frac{p-3}{p}}ds \\
&\leq t^{-\frac{1+\alpha}{2}}K(t)\cdot V(t).
\end{align*}
Then for any $t\in[0,T]$, we get \eqref{3.33}.
\end{proof}

\begin{lemma}\label{Lemma3.6}
Under the arguments of Lemma \ref{Lemma3.1}--Lemma \ref{lemma3.7}, we obtain
\begin{align}
t^{\frac{1}{2}}\|u(t)\|_{L^{\infty}(\mathbb{R}^{3})}\leq P(K_0(t),V(t),\alpha).\label{3.35}
\end{align}
where $0<\alpha<1$, $P(s)$ denotes the polynomial function related to $s$ .
\end{lemma}

\begin{proof}
Applying the interpolation inequality, we have
\begin{align*}
\|u(t)\|_{L^{\infty}(\mathbb{R}^3)}\leq\|\nabla u(x,t)\|^{\frac{1}{1+\alpha}}_{M^{\frac{3}{1-\alpha},3}(\mathbb{R}^{3})}
\|u(t)\|^{\frac{\alpha}{\alpha+1}}_{M^{3,3}(\mathbb{R}^{3})}.
\end{align*}
Since ${\sup\limits_{\tau\in(0,T)}}(K_0(t)+V(t))<\rho$, by Lemma \ref{pro3.1} and Lemma \ref{Lemma3.2}, one has
\begin{align*}
\|u(t)\|_{M^{3,3}(\mathbb{R}^{3})}\leq P(K_0(t),V(t),\alpha).
\end{align*}
Then, we can derive that
\begin{align*}
t^{\frac{1}{2}}\|u(t)\|_{{L^{\infty}}(\mathbb{R}^3)}\leq \left(t^{\frac{1+\alpha}{2}}
\|\nabla u(x,t)\|^{\frac{1}{1+\alpha}}_{M^{\frac{3}{1-\alpha},3}(\mathbb{R}^{3})}\right)^{\frac{1}{1+\alpha}} P(K_0(t),V(t),\alpha).
\end{align*}
Combining the above arguments with Lemma \ref{lemma3.7} and Proposition \ref{pro3.1}, we get the following desired result:
\begin{align*}
t^{\frac{1}{2}}\|u(t)\|_{L^{\infty}(\mathbb{R}^3)}\leq P(K_0(t),V(t),\alpha).
\end{align*}
We complete the proof of Lemma \ref{Lemma3.6}.
\end{proof}

\section{The regularity theory}\label{Sec6}
According to the previous analysis, we get the existence of local smooth solutions with high regularity of the initial data. At this time, it is easy to extend local solutions to global solutions by Proposition \ref{global-smooth-solu} and Lemma \ref{Lemma3.6}. In this section, we prove the existence of the global smooth solution with the initial data as main theorem, namely $\boldsymbol{m}_{0}-\boldsymbol{m}_{\infty}\in M^{2,3,1}(\mathbb{R}^{3})\cap M^{3,3,1}(\mathbb{R}^{3})$. It is obviously that the method of proving the existence of local smooth solutions is invalid. In subsection \ref{Sec6.1}, we reduce the regularity of the initial data and get a limitation of a sequence of global smooth solutions in terms of an approximating result. In subsection \ref{Sec6.2}, we improve the regularity of the limitation result. We shall use the following approximation result, which is originally given by Schoen and Uhlenbeck in \cite{Scho} and \cite[Proposition 7.2]{Stu}:
\begin{lemma}\label{lemma appro}
Assume that $\boldsymbol{m}:\mathbb{R}^{34}\rightarrow\mathbb{S}^{2}$ and $\boldsymbol{m}-\boldsymbol{m}_{\infty}\in M^{2,3,1}(\mathbb{R}^{3})\cap M^{3,3,1}(\mathbb{R}^{3})$, then there exists a sequence $\boldsymbol{m}^{k}\in M^{2,3,\infty}_{\ast}(\mathbb{R}^{3};\mathbb{S}^{2})$ satisfies
\begin{align*}
\boldsymbol{m}-\boldsymbol{m}^{k}\rightarrow 0~~\in M^{2,3,1}(\mathbb{R}^{3})\cap M^{3,3,1}(\mathbb{R}^{3}).
\end{align*}
Moreover, for any integer $\sigma>1$, if $\boldsymbol{m}\in M^{2,3,\sigma}_{\ast}(\mathbb{R}^{3};\mathbb{S}^{2})$, then $\{\boldsymbol{m}^{k}\}$ is uniformly bounded in  $M^{2,3,\sigma}_{\ast}(\mathbb{R}^{n};\mathbb{S}^{2})$.
\end{lemma}
\subsection{Global weak solution}\label{Sec6.1}
For any given initial data $\boldsymbol{m}_{0}:\mathbb{R}^{3}\rightarrow\mathbb{S}^{2}$ and $\boldsymbol{m}_{0}-\boldsymbol{m}_{\infty}\in M^{2,3,1}\cap M^{3,3,1}(\mathbb{R}^{3})$, applying Lemma \ref{lemma appro}, we see that there exists a sequence $\boldsymbol{m}_{0}^{k}\in M^{2,3,\infty}_{\ast}(\mathbb{R}^{n};\mathbb{S}^{2})$ satisfies
\begin{align*}
\|\boldsymbol{m}_{0}-\boldsymbol{m}_{0}^{k}\|_{M^{2,3,1}\cap M^{3,3,1}(\mathbb{R}^{3})}\rightarrow 0\qquad (k\rightarrow\infty).
\end{align*}
At this time, the corresponding solutions are local and smooth solution by Proposition \ref{global-smooth-solu}, i.e., $\boldsymbol{m}^{k}\in C^{0}([0, T);M^{2,3,\infty}_{\ast}(\mathbb{R}^{3};\mathbb{S}^{2}))$.

Next, we extend local solutions $\boldsymbol{m}^{k}\in C^{0}([0, T);M^{2,3,\infty}_{\ast}(\mathbb{R}^{3};\mathbb{S}^{2}))$ to global solutions. The key point is to prove that $\|\nabla \boldsymbol{m}^{k}\|_{L^{\infty}(\mathbb{R}^{3})}$ is bounded. Since $|\nabla\boldsymbol{m}^{k}|=|u^{k}(t)|$, it is equivalent to prove that $\|u^{k}(t)\|_{L^{\infty}(\mathbb{R}^{3})}\leq \infty$. Suppose that $u^{k}(t)$ is the smooth solution of \eqref{2.10}, by Lemma \ref{Lemma3.2} and Lemma \ref{Lemma3.6}, one has
\begin{align*}
\|u^{k}(t)\|_{M^{3,3}(\mathbb{R}^{3})}+t^{\frac{1}{2}}\|u^{k}(t)\|_{{L^{\infty}}(\mathbb{R}^3)} \leq P(K^{k}_{0}(t),V^{k}(t),\alpha),
\end{align*}
where $K^{k}_{0}(t),V^{k}(t)$ represent the variables related to $u^{k}$ and $v_{k}$. By Proposition \ref{pro3.1} if $\sup_{t\in(0,T)}(K^{k}_{0}(t)+V^{k}(t))<\rho$, for any $t\in[0,T]$, then $K^{k}(t)\leq2K^{k}_{0}(t)$. Therefore, we have
\begin{align*}
P(K^{k}_{0}(t),V^{k}(t),\alpha) \leq  K^{k}_{0}(t)\leq \rho,
\end{align*}
which yields that
\begin{align*}
\|u^{k}(t)\|_{M^{3,3}(\mathbb{R}^{3})}+t^{\frac{1}{2}}\|u^{k}(t)\|_{{L^{\infty}}(\mathbb{R}^3)} \leq K^{k}_{0}(t)+V^{k}(t).
\end{align*}
Now if $\boldsymbol{m}^{k}\in M^{2,3,\infty}(\mathbb{R}^{3})$ are local smooth solutions of \eqref{1.1}, we notice that $|\nabla \boldsymbol{m}^{k}|=|u^{k}(t,x)|$ by moving frames, and then we have
\begin{align*}
\sup_{t>0}\|\nabla \boldsymbol{m}^{k}(t)\|_{M^{3,3}(\mathbb{R}^{3})}+\sqrt{t}\sup_{t>0}\|\nabla \boldsymbol{m}^{k}\|_{{L^{\infty}}(\mathbb{R}^3)} \leq \|\nabla \boldsymbol{m}^{k}_{0}\|_{M^{3,3}(\mathbb{R}^{3})}.
\end{align*}
Hence for any $t>0$, we can extend the local solution to the global solution without loss of regularity by Proposition \ref{global-smooth-solu}, namely $\boldsymbol{m}^{k}\in C^{0}((0, \infty);M^{2,3,\infty}_{\ast}(\mathbb{R}^{3};\mathbb{S}^{2}))$.

Moreover, a natural idea is to consider whether the limitation of $\boldsymbol{m}^{k}$ (which is still denoted by $\boldsymbol{m}$ for simplicity) is also a global smooth solution. It seems very challenging for verifying directly $\boldsymbol{m}$ is smooth and global, with the fact that it is equivalent to show
$\lim_{k\rightarrow\infty}\|\boldsymbol{m}^{k}-\boldsymbol{m}\|_{M^{2,3,\sigma}_{\ast}(\mathbb{R}^{3})}\rightarrow0, (\sigma=1,2,...)$. Therefore, we firstly prove that $\boldsymbol{m}$ is a global weak solution, and then improve the regularity of the weak solution to gain a global smooth solution.

The inequality \eqref{3.2} implies that
\begin{align*}
\|\boldsymbol{m}_{1}^{k}-\boldsymbol{m}_{2}^{k}\|_{M^{2,3}(\mathbb{R}^{3})}^{2}+\int_{0}^{t}
\|\nabla\boldsymbol{m}_{1}^{k}-\nabla\boldsymbol{m}_{2}^{k}\|_{M^{2,3}(\mathbb{R}^{3})}^{2}ds
\leq \|\boldsymbol{m}_{1}^{k}(0)-\boldsymbol{m}_{2}^{k}(0)\|_{M^{2,3}(\mathbb{R}^{3})}^{2},
\end{align*}
by the fact that $\|\boldsymbol{m}_{0}-\boldsymbol{m}_{0}^{k}\|_{M^{2,3,1}\cap M^{3,3,1}(\mathbb{R}^{3})}\rightarrow0$ and $\{\boldsymbol{m}^{k}_{0}\}_{k=1}^{\infty}$ is the Cauchy sequence. Then, it holds that $\|\boldsymbol{m}_{1}^{k}(0)-\boldsymbol{m}_{2}^{k}(0)\|_{M^{2,3}(\mathbb{R}^{3})}^{2}\rightarrow0$. Obviously it easy to obtain that $\boldsymbol{m}^{k}$ converges to $\boldsymbol{m}$ in $M^{2,3,1}(\mathbb{R}^{3})$ . By the definition of weak solutions, choosing a test function $\phi\in C^{\infty}_{0}(\mathbb{R}^{3})$ and multiplying \eqref{main-eq} by $\phi$, we get
{\small
\begin{align*}
\langle \frac{\partial\boldsymbol{m}^{k}}{\partial t},\phi\rangle +\langle(v\cdot\nabla)\boldsymbol{m}^{k}
+\boldsymbol{m}^{k}\times(v\cdot\nabla)\boldsymbol{m}^{k},\phi\rangle=\langle\boldsymbol{m}^{k}\times\nabla\boldsymbol{m}^{k},\nabla\phi\rangle
+\varepsilon\langle\Delta\boldsymbol{m}^{k}+|\nabla\boldsymbol{m}^{k}|^{2}\boldsymbol{m}^{k},\phi\rangle.
\end{align*}}
Since $\boldsymbol{m}^{k}\rightarrow \boldsymbol{m}$ strongly in $M^{2,3,1}(\mathbb{R}^{3})$, we have
\begin{align*}
\langle \frac{\partial\boldsymbol{m}}{\partial t},\phi\rangle +\langle(v\cdot\nabla)\boldsymbol{m},\phi\rangle
+\langle \boldsymbol{m}\times(v\cdot\nabla)\boldsymbol{m},\phi\rangle=\langle\boldsymbol{m}\times\nabla\boldsymbol{m},\nabla\phi\rangle
+\varepsilon\langle\Delta\boldsymbol{m}+|\nabla\boldsymbol{m}|^{2}\boldsymbol{m},\phi\rangle.
\end{align*}
By the definition of weak solutions, $\boldsymbol{m}$ is the global weak solution with the initial data $\boldsymbol{m}_{0}-\boldsymbol{m}_{\infty}\in M^{2,3,1}\cap M^{3,3,1}(\mathbb{R}^{3})$.

\subsection{Global smooth solution}\label{Sec6.2}
In this subsection, we want to use the regularity theory to prove the regularity of $\boldsymbol{m}$. The main key in showing global regularity of weak solutions is to prove that the boundedness of $\boldsymbol{m}$ in a H\"{o}lder space which is isomorphic to a functional space. Notice that the feature of the functional space is similar to the BMO space, which implies that we will take into account of an estimate of $\boldsymbol{m}$ in the BMO space. Before one can reach that conclusion, we need a vital inequality that is the estimate of the $M^{2,3}$-norm of $\partial_{t}\boldsymbol{m}$ can be controlled by the $M^{2,3}$-norm of $\nabla\boldsymbol{m}$ in appropriate cylinders.
The other important point is that the following monotonicity property: for any $z=(x,t)\in \mathbb{R}^{3}\times \mathbb{R}^{1}_{+}$, the function
\begin{align*}
F_{z}(r)=:E(r,z):=\frac{1}{r^{n}}\int_{P_{r}(z)}|\nabla\boldsymbol{m}|^{2}dxdt.
\end{align*}
is nondecreasing, which implies a corresponding small energy condition about $E(z,r)$. Firstly, we give some useful lemmas and corresponding conclusions mentioned above.
\begin{lemma}\label{Im inequality}
For any $\alpha\in(0,1)$, $z\in\mathbb{R}^{3}\times\mathbb{R}^{1}_{+}, r>0$, there exists a constant $C_{\alpha}=C(\alpha)$ such that
\begin{align*}
\int_{p_{\alpha r(z)}}\left|\frac{\partial \boldsymbol{m}}{\partial t}\right|^{2}dxdt
\lesssim \frac{C_\alpha}{r^{2}}\int_{p_{r(z)}}|\nabla \boldsymbol{m}|^{2}dxdt
\end{align*}
\end{lemma}

\begin{proof}
Considering $\boldsymbol{m}\times\eqref{main-eq}+\eqref{main-eq}$, and then omitting some constants for simplicity, we have
\begin{align}\label{main-eq2}
\begin{split}
{\partial_t}\boldsymbol{m}+\boldsymbol{m}\times{\partial_t}\boldsymbol{m}
+\boldsymbol{m}\times(\boldsymbol{v}\cdot\nabla)\boldsymbol{m}
=\Delta\boldsymbol{m}+|\nabla\boldsymbol{m}|^{2}\boldsymbol{m}.
\end{split}
\end{align}
Without loss of generality, for any $P_{r}(z)\subseteq P_{2}(0,4)$, we choose a test function $\phi(\xi)\in C^{\infty}_{0}(\mathbb{R}^{n})$ satisfying $\phi(\xi)=1$ in $B_{\alpha r}$ and $|\nabla\phi(\xi)|\leq \frac{c}{r}$. Multiplying  \eqref{main-eq2} by ${\partial_t}\boldsymbol{m}\phi^{2}$, with the fact that $|\boldsymbol{m}|=1$,  it holds that
\begin{align*}
\langle{\partial_t}\boldsymbol{m},{\partial_t}\boldsymbol{m}\phi^{2} \rangle
+\langle\boldsymbol{m}\times(\boldsymbol{v}\cdot\nabla)\boldsymbol{m},{\partial_t}\boldsymbol{m}\phi^{2}\rangle
=\langle\Delta\boldsymbol{m},{\partial_t}\boldsymbol{m}\phi^{2}\rangle.
\end{align*}
By using H\"{o}lder's inequality and Young's inequality, one has
\begin{align}\label{2r}
\begin{split}
&\int_{B_{r}}|{\partial_t}\boldsymbol{m}\phi|^{2} dx\\
&\leq\|\boldsymbol{v}\nabla\boldsymbol{m}{\partial_t}\boldsymbol{m}\phi^{2}\|_{M^{1,3}(B_{r})}
+C(\delta)\|\nabla\boldsymbol{m}\nabla\phi\|^{2}_{M^{2,3}(B_{r})}
+\delta\|{\partial_t}\boldsymbol{m}\phi\|^{2}_{M^{2,3}(B_{r})}
+\int_{B_{r}}{\partial_{t}(\nabla\boldsymbol{m})^{2}}\phi^{2}dx\\
&\lesssim\|\boldsymbol{v}\nabla\boldsymbol{m}{\partial_t}\boldsymbol{m}\phi^{2}\|_{M^{1,3}(B_{r})}
+C(\delta)\frac{1}{r^{2}}\|\nabla\boldsymbol{m}\|^{2}_{M^{2,3}(B_{r})}
+\int_{B_{r}}{\partial_{t}(\nabla\boldsymbol{m})^{2}}\phi^{2}dx.
\end{split}
\end{align}
Next, we estimate the first term on the right-hand side of \eqref{2r}. It is worthy that the definition of $\nabla\boldsymbol{m}\phi$ is in the whole space, so we can use the Fourier transform. By the Plancherel equality, Young's inequality, and then using the Gagliardo-Nirenberg-Sobolev inequality, one has
\begin{align}\label{1.3}
\begin{split}
&\|\boldsymbol{v}\nabla\boldsymbol{m}{\partial_t}\boldsymbol{m}\phi^{2}\|_{M^{1,3}(B_{r})}\\
&\leq \delta\|{\partial_t}\boldsymbol{m}\phi \boldsymbol{v}\|^{2}_{M^{2,3}(B_{r})}
+C(\delta)\|\nabla\boldsymbol{m}\phi\|^{2}_{M^{2,3}(B_{r})}\\
&\leq \delta\|{\partial_t}\boldsymbol{m}\phi\|^{2}_{M^{2,3}(B_{r})}\|\boldsymbol{v}\|^{2}_{L^{\infty}(B_{r})}
+C(\delta)\|(i\xi)\hat{\boldsymbol{m}}\ast\hat{\phi}\|^{2}_{M^{2,3}_{\xi}(\mathbb{R}^{n})}\\
&\leq\delta\|{\partial_t}\boldsymbol{m}\|^{2}_{M^{2,3}(B_{r})}\|\boldsymbol{v}\|^{2}_{L^{\infty}(\mathbb{R}^{n})}
+C(\delta)\|\nabla\boldsymbol{m}\|^{2}_{M^{2,3}(B_{r})}\|\hat{\phi}\|^{2}_{M^{1,3}(\mathbb{R}^{n})}\\
&\leq\delta\|{\partial_t}\boldsymbol{m}\|^{2}_{M^{2,3}(B_{r})}\|\boldsymbol{v}\|^{2}_{L^{\infty}(\mathbb{R}^{n})}
+C(\delta)\|\nabla\boldsymbol{m}\|^{2}_{M^{2,3}(B_{r})}\|\phi\|^{4}_{M^{2,3}(B_{r})}\\
&\leq\delta\|{\partial_t}\boldsymbol{m}\|^{2}_{M^{2,3}(B_{r})}\|\boldsymbol{v}\|^{2}_{L^{\infty}(\mathbb{R}^{n})}
+C(\delta)\|\nabla\boldsymbol{m}\|^{2}_{M^{2,3}(B_{r})}\|\nabla\phi\|^{\frac{8}{3}}_{M^{1,3}(B_{r})}
\|\phi\|^{\frac{4}{3}}_{M^{6,3}(B_{r})}\\
&\leq\delta\|{\partial_t}\boldsymbol{m}\|^{2}_{M^{2,3}(B_{r})}\|\boldsymbol{v}\|^{2}_{L^{\infty}(\mathbb{R}^{n})}
+C(\delta)\|\nabla\boldsymbol{m}\|^{2}_{M^{2,3}(B_{r})}\|\nabla\phi\|^{\frac{8}{3}}_{M^{1,3}(B_{2}(0))}
\|\phi\|^{\frac{4}{3}}_{M^{6,3}(B_{r})}\\
&\leq\delta\|{\partial_t}\boldsymbol{m}\|^{2}_{M^{2,3}(B_{r})}\|\boldsymbol{v}\|^{2}_{L^{\infty}(\mathbb{R}^{n})}
+C(\delta)\|\nabla\boldsymbol{m}\|^{2}_{M^{2,3}(B_{r})}\left|\frac{1}{r}\right|^{\frac{8}{3}}|B_{r}|^{\frac{4}{3}\cdot\frac{1}{6}}\\
&\lesssim\delta\|{\partial_t}\boldsymbol{m}\|^{2}_{M^{2,3}(B_{r})}\|\boldsymbol{v}\|^{2}_{L^{\infty}(\mathbb{R}^{n})}
+C(\delta)\frac{1}{r^{2}}\|\nabla\boldsymbol{m}\|^{2}_{M^{2,3}(B_{r})}.
\end{split}
\end{align}
Let $\delta'=\delta\|\boldsymbol{v}\|^{2}_{L^{\infty}(\mathbb{R}^{n})}$, we select a suitable $\delta$ such that $\delta'$ is sufficiently small to absorb the term $\|{\partial_t}\boldsymbol{m}\|^{2}_{M^{2,3}(B_{r})}$.

Observing that the third term on the right-hand side of \eqref{2r}, we can add an integral with respect to time on $(s,t)$ to get
\begin{align}\label{1.1}
\int_{s}^{t}\int_{B_{r}(z)}{\partial_{t}(\nabla\boldsymbol{m})^{2}}dxdt
\leq \int_{B_{r}(z)}(\nabla\boldsymbol{m})^{2}(t)dx-\int_{B_{r}(z)}(\nabla\boldsymbol{m})^{2}(s)dx.
\end{align}
Inspired by \cite[Lemma 2.2]{DingWang}, we can derive that
\begin{align}\label{1.2}
\int_{B_{r}(z)}(\nabla\boldsymbol{m})^{2}(t)dx
\lesssim \frac{1}{r^{2}}\int_{P_{r}(z)}(\nabla\boldsymbol{m})^{2}dxdt.
\end{align}
Then, substituting \eqref{1.3}--\eqref{1.2} into \eqref{2r}, we get
\begin{align*}
\int_{p_{\alpha r(z)}}\left|\frac{\partial \boldsymbol{m}}{\partial t}\right|^{2}dxdt
\lesssim \frac{C_{\alpha}}{r^{2}}\int_{p_{r(z)}}|\nabla \boldsymbol{m}|^{2}dxdt.
\end{align*}
\end{proof}

\begin{lemma}\label{lemma BMO es}
For any $\alpha\in(0,1)$, $z\in\mathbb{R}^{3}\times\mathbb{R}^{1}_{+}, r>0$, there exists a constant $C_{\alpha}$ such that
\begin{align*}
\int_{p_{\alpha r(z)}}|\boldsymbol{m}(x,t)-(\boldsymbol{m})_{z,\alpha r}|^{2}dxdt
\leq C_{\alpha}r^{2}\int_{p_{r(z)}}|\nabla \boldsymbol{m}|^{2}dxdt.
\end{align*}
\end{lemma}

\begin{proof}
Let $w(x,t)=\boldsymbol{m}(x_{0}+rx,t_{0}+r^{2}t)-(\boldsymbol{m})_{z_{0},\alpha r}$, where $\boldsymbol{m}\in H^{1}_{loc}(\mathbb{R}^{3}\times\mathbb{R}^{1})$, $z_{0}=(x_{0},t_{0})\in\mathbb{R}^{n}\times\mathbb{R}^{1}_{+}, r>0$.
By Poincar\'{e}'s inequality, we have
\begin{align*}
\int_{p_{\alpha}(0)}|w-(w)_{0,\alpha}|^{2}dxdt \leq C_{\alpha} \int_{p_{\alpha}(0)}|\nabla w|^{2}+w_{t}^{2}dxdt.
\end{align*}
Applying the definition of $w(x,t)$ implies that $(w)_{0,\alpha}=0$, then one has
\begin{align*}
&\int_{p_{\alpha}(0)}|\boldsymbol{m}(x_{0}+rx,t_{0}+r^{2}t)-(\boldsymbol{m})_{z_{0},\alpha r}|^{2}dxdt\\
&\leq C_{\alpha} \int_{p_{\alpha}(0)}|\nabla \boldsymbol{m}(x_{0}+rx,t_{0}+r^{2}t)|^{2}+\boldsymbol{m}^{2}_{t}(x_{0}+rx,t_{0}+r^{2}t)dxdt.
\end{align*}
We use the coordinate transform to obtain that
\begin{align*}
\int_{p_{\alpha r}(z_{0})}|\boldsymbol{m}(x,t)-(\boldsymbol{m})_{z_{0},\alpha r}|^{2}dxdt
\leq C_{\alpha}r^{2}\int_{p_{\alpha r}(z_{0})}|\nabla \boldsymbol{m}(x,t)|^{2}+r^{2}\boldsymbol{m}^{2}_{t}(x,t)dxdt.
\end{align*}
Finally, by Lemma \ref{Im inequality}, one has
\begin{align*}
\int_{p_{\alpha r}(z_{0})}|\boldsymbol{m}(x,t)-(\boldsymbol{m})_{z_{0},\alpha r}|^{2}dxdt
\leq C_{\alpha}r^{2}\int_{p_{r}(z_{0})}|\nabla \boldsymbol{m}(x,t)|^{2}dxdt.
\end{align*}
\end{proof}

\begin{lemma}\label{E(z,r)}
Assume that $\boldsymbol{m}$ is a weak solution of \eqref{main-eq}, if there exist a constant $c_{1}>0$ and a sufficiently small constant $\varepsilon_{0}>0$ such that for any $T_{0}>0,~\tau\in (0,1),~z\in \mathbb{R}^{n}\times[T_{0},\infty),r<\frac{\sqrt{T_{0}}}{2}$, $E(z,r)<\varepsilon_{0}$, then we have $E(z,\tau r)\leq 2c_{1}\tau^{2}E(z,r)$.
\end{lemma}

\begin{proof}
We argue by contradiction. Suppose that there exists a sequence $\{(z_{k},r_{k})\}$, where $z_{k}\in \mathbb{R}^{n}\times [T_{0},\infty), 0<r_{k}<\frac{\sqrt{T_{0}}}{2}$ such that
\begin{align*}
E(z_{k},r_{k})=\lambda_{k}^{2}\rightarrow0,~as~k\rightarrow\infty.
\end{align*}
and
\begin{align*}
E(z_{k},\tau r_{k})>2c_{1}\tau^{2}\lambda_{k}^{2},\quad \tau\in(0,1).
\end{align*}
We define the functions $h_{k}(y,s)=\frac{\boldsymbol{m}(x_{k}+yr_{k},t_{k}+sr^{2}_{k})-a_{k}}{\lambda_{k}}$,
where $a_{k}={\int\hspace{-0.90em}-}_{P_{\frac{9}{10}r_{k}}(z_{k})} \boldsymbol{m} dz$. Employing Lemma \ref{Im inequality}, we have
\begin{align*}
\lambda_{k}^{2}\int_{P_{\frac{9}{10}}}|h_{k,t}|^{2}dz=\frac{1}{r_{k}^{n-2}}\int_{P_{\frac{9}{10}r_{k}(z_{k})}}|\boldsymbol{m}_{t}|^{2}dz
\leq \frac{1}{r_{k}^{n}}\int_{P_{\frac{9}{10}r_{k}(z_{k})}}|\nabla\boldsymbol{m}|^{2}dz
\leq E(z_{k},r_{k})=\lambda_{k}^{2}.
\end{align*}
Applying the coordinate transform and Lemma \ref{lemma BMO es}, we easily see that
\begin{align}\label{contradictory}
\frac{1}{\tau^{n}}\int_{P_{\tau}}|\nabla h_{k}|^{2}dz
=\frac{1}{\tau^{n}r^{n}_{k}\lambda_{k}^{2}}\int_{P_{\tau r_{k}}}|\nabla \boldsymbol{m}|^{2}dz
=\frac{1}{\lambda_{k}^{2}}E(z_{k},\tau r_{k})>2c_{1}\tau^{2},
\end{align}
\begin{align*}
\sup_{k}\int_{P_{\frac{9}{10}}}|h_{k,t}|^{2}dz\leq C,
\end{align*}
\begin{align*}
\int_{P_{1}(0)}|\nabla h_{k}|^{2}dz=\frac{1}{\lambda_{k}^{2}r_{k}^{n}}\int_{P_{r_{k}(z_{k})}}|\nabla\boldsymbol{m}|^{2}dz
=\frac{1}{\lambda_{k}^{2}}E(z_{k},r_{k})=1,
\end{align*}
\begin{align*}
\sup_{k}\int_{P_{\frac{9}{10}}}|h_{k}|^{2}dz=&\sup_{k}\frac{1}{\lambda_{k}^{2}r_{k}^{n+2}}\int_{P_{\frac{9}{10}r_{k}(z_{k})}}
|\boldsymbol{m}(x,t)-a_{k}|^{2}dz
\leq \frac{1}{\lambda_{k}^{2}r_{k}^{n}}\int_{P_{r_{k}(z_{k})}}|\nabla\boldsymbol{m}|^{2}dz
\leq C,
\end{align*}
where $C$ is a positive constant independently of $\lambda_{k}$ and $z_{k}$.
Hence, we can obtain the sequence $\{h_{k}\}_{k=1}^{\infty}$ is bounded in $M^{2,3,1}(P_{\frac{9}{10}};\mathbb{R}\times\mathbb{R}^{3})$, and then by the weak compactness which implies that there exists a subsequence $\{h_{k_{i}}\}_{i=1}^{\infty}$ (it is still written as $\{h_{k}\}$) such that
\begin{align}\label{weak conver}
\begin{split}
h_{k}\rightharpoonup h~in~M^{2,3}(P_{\frac{9}{10}};\mathbb{R}\times\mathbb{R}^{3}),\\
\nabla h_{k}\rightharpoonup  \nabla h~in~M^{2,3}(P_{1};\mathbb{R}\times\mathbb{R}^{3}),\\
h_{k,t}\rightharpoonup h_{t}~in~M^{2,3}(P_{\frac{9}{10}};\mathbb{R}\times\mathbb{R}^{3}).
\end{split}
\end{align}
Using the fact that $\{h_{k}\}\in M^{2,3,1}(P_{\frac{9}{10}};\mathbb{R}\times\mathbb{R}^{3})\subset\subset M^{2,3}(P_{\frac{9}{10}};\mathbb{R}\times\mathbb{R}^{3})$, we see that
\begin{align}\label{strong conver}
h_{k}\rightarrow h~in~M^{2,3}(P_{\frac{9}{10}};\mathbb{R}\times\mathbb{R}^{3}).
\end{align}
Next, we select a smooth function with compact support $w:P_{\frac{9}{10}}\rightarrow\mathbb{R}\times\mathbb{R}^{3}$ and define functions $w_{k}:P_{\frac{9}{10}r_{k}(z_{k})}\rightarrow\mathbb{R}\times\mathbb{R}^{3}$ satisfying
\begin{align*}
w_{k}(y,s)=w(\frac{y-x_{k}}{r_{k}},\frac{s-t_{k}}{r_{k}^{2}}).
\end{align*}
Consider
\begin{align*}
&\langle{\partial_s}\boldsymbol{m},w_{k}\rangle_{P_{r_{k}(z_{k})}}
+\langle\boldsymbol{m}\times{\partial_s}\boldsymbol{m},w_{k}\rangle_{P_{r_{k}(z_{k})}}
+\langle\boldsymbol{m}\times(\boldsymbol{v}\cdot\nabla)\boldsymbol{m},w_{k}\rangle_{P_{r_{k}(z_{k})}}\\
&=\langle\Delta\boldsymbol{m},w_{k}\rangle_{P_{r_{k}(z_{k})}}
+\langle|\nabla\boldsymbol{m}|^{2}\boldsymbol{m},w_{k}\rangle_{P_{r_{k}(z_{k})}}.
\end{align*}
Rescaling the domain of integration to the unit cylinder $x=\frac{y-x_{k}}{r_{k}},~t=\frac{s-t_{k}}{r_{k}^{2}}$, and then replacing $h_{k}$ with $\boldsymbol{m}$, it holds that
\begin{align}\label{hk-eq}
\begin{split}
&\int_{P_{\frac{9}{10}}}{\partial_t}h_{k}\cdot w dxdt
\quad+\int_{P_{\frac{9}{10}}}(\lambda_{k}h_{k}+a_{k})\times{\partial_t}h_{k}\cdot w dxdt\\
&+r_{k}\int_{P_{\frac{9}{10}}}(\lambda_{k}h_{k}+a_{k})\times[v(x_{k}+xr_{k},t_{k}+tr_{k}^{2})\cdot \nabla]h_{k}\cdot w dxdt\\
&=-\int_{P_{\frac{9}{10}}}\nabla h_{k}:\nabla w dxdt
+\lambda_{k}\int_{P_{\frac{9}{10}}}|\nabla h_{k}|^{2}(\lambda_{k}h_{k}+a_{k})\cdot w dxdt.
\end{split}
\end{align}
We know that $\lambda_{k}\rightarrow0$, $a_{k}, r_{k}$ are constant, $v,r_{k}$ is bounded function, and then combining \eqref{weak conver} with \eqref{strong conver}, one has
\begin{align}\label{h-eq}
\int_{P_{\frac{9}{10}}}{\partial_t}h\cdot w dxdt+\int_{P_{\frac{9}{10}}}\nabla h:\nabla w dxdt=0.
\end{align}
Then, we can deduce that $h(t,x)\in M^{2,3,1}(P_{\frac{9}{10}}, \mathbb{R}\times\mathbb{R}^{3})$ is the solution of ${\partial_{t}}h-\Delta h=0$ in the weak sense and therefore in classical sense, hence $h$ is smooth inside $P_{\frac{9}{10}}$.
By the Caccioppoli inequality, one has
\begin{align}\label{E-Dh-bdd}
\begin{split}
\frac{1}{\tau^{n}}\int_{P_{\tau}}|\nabla h|^{2}dxdt
&\leq \frac{1}{\tau^{n-2}}\int_{P_{\tau}}h^{2}dxdt
\leq \tau^{2}{\int\hspace{-1.05em}-}_{P_{\tau}} h^{2} dxdt \\
&=\tau^{2}{\int\hspace{-1.05em}-}_{P_{\frac{9}{10}}}h^{2}dxdt
\leq c_{1}\tau^{2},
\end{split}
\end{align}
and then using Lemma \ref{strong-conver-th}, we know that
\begin{align}\label{Dhk convergence}
\nabla h_{k}\rightarrow \nabla h \text{ strongly in }M^{2,3}(P_{\frac{1}{2}}).
\end{align}
By the convergence \eqref{E-Dh-bdd} and \eqref{Dhk convergence}, we have
\begin{align*}
\frac{1}{\tau^{n}}\int_{P_{\tau}}|\nabla h_{k}|^{2}dyds
\leq c_{1}\tau^{2}.
\end{align*}
Substituting the definition of $h_{k}$ into the above inequality, we get
\begin{align*}
E(z_{k},\tau r_{k})\leq c_{1}\tau^{2}\lambda_{k}^{2},\quad \tau\in(0,1),
\end{align*}
which is contradictory to \eqref{contradictory}. Therefore, the assumption is wrong! So we have $E(z,\tau r)\leq 2c_{1}\tau^{2}E(z,r)$.
\end{proof}

\begin{remark}
We notice that ${\int\hspace{-0.90em}-}_{P_{r}(z_{0})}|h_{k}-(h_{k})_{z_{0},r}|dxdt\leq C$ by Poincar\'{e}'s inequality, i.e., $h_{k}\in BMO(P_{r}(z_{0}))$. Applying the John-Nirenberg inequality, we have
\begin{align*}
\{h_{k}\}_{k=1}^{\infty}~~is~~ bounded~~in~~M^{p,3}(P_{\frac{9}{10}}),\;1<q<\infty.
\end{align*}
By the Kondrachov-Rellich theorem we know that $M^{2,3,1}\subset\subset M^{q,3}$. Combining this with the weak compactness, similar to \eqref{strong conver}, one has
\begin{align}\label{hk-L4-strong}
h_{k}\rightarrow h~in~M^{p,3}(P_{\frac{9}{10}};\mathbb{R}\times\mathbb{R}^{3}),\;1<q<6.
\end{align}
\end{remark}

\begin{lemma}\label{strong-conver-th}
The sequence $\{\nabla h_{k}\}_{k=1}^{\infty}$ is strongly convergent in $M^{2,3}(P_{\frac{1}{2}})$,
i.e.,
\begin{align*}
\{\nabla h_{k}\}_{k=1}^{\infty} \rightarrow \nabla h~in~M^{2,3}(P_{\frac{1}{2}});\mathbb{R}\times\mathbb{R}^{n}).
\end{align*}
\end{lemma}

\begin{proof}
Subtracting \eqref{hk-eq} from \eqref{h-eq}, one has
\begin{align}\label{hk-h-eq}
\begin{split}
&\int_{P_{1}}({\partial_t}h_{k}-{\partial_t}h)\cdot w +(\nabla h_{k}-\nabla h):\nabla w dxdt\\
&=\lambda_{k}\int_{P_{1}}|\nabla h_{k}|^{2}(\lambda_{k}h_{k}+a_{k})\cdot w dxdt
-\int_{P_{1}}(\lambda_{k}h_{k}+a_{k})\times{\partial_t}h_{k}\cdot w dxdt\\
&\quad-r_{k}\int_{P_{1}}(\lambda_{k}h_{k}+a_{k})\times[v(x_{k}+xr_{k},t_{k}+tr_{k}^{2})\cdot \nabla]h_{k}\cdot w dxdt\\
&=:R_{1}+R_{2}+R_{3}.
\end{split}
\end{align}
We find that the above equation still hold about $w\in M^{2,3,1}_{0}(P_{\frac{9}{10}})\cap L^{\infty}(P_{\frac{9}{10}})$ by approximation. Let $w=\zeta^{2}(h_{k}-h)$, $\zeta\in C^{\infty}_{c}(P_{\frac{9}{10}})$, the terms on the left-hand side of \eqref{hk-h-eq} are written as $L$. When $k\rightarrow 0$, it holds that
\begin{align*}
L&=\int_{P_{\frac{9}{10}}}\zeta^{2}|\nabla h_{k}-\nabla h|^{2}dxdt
+2\int_{(P_{\frac{9}{10}})}\zeta(h_{k}-h)(\nabla h_{k}-\nabla h)\nabla\zeta dxdt\\
&+\int_{P_{\frac{9}{10}}}\zeta^{2}({\partial_t}h_{k}-{\partial_t}h)(h_{k}-h)dxdt\\
&\geq \int_{P_{\frac{1}{2}}}|\nabla h_{k}-\nabla h|^{2}dxdt+o(1),
\end{align*}
where $o(1)$ can be obtained by \eqref{strong conver}.

In the following, we estimate the terms $R_{2}$ and $R_{3}$. Applying H\"{o}lder's inequality, \eqref{weak conver}, \eqref{strong conver}, and \eqref{hk-L4-strong}, as $k\rightarrow \infty$ we see that
\begin{align*}
R_{2}
&=\lambda_{k}\int_{P_{\frac{9}{10}}}h_{k}\times{\partial_t}h_{k}\cdot \zeta^{2}(h_{k}-h) dxdt\\
&\leq \lambda_{k}\|\zeta\|^{2}_{L^{\infty}(P_{\frac{9}{10}})}\|h_{k}-h\|_{M^{4,3}(P_{\frac{9}{10}})}
\|h_{k}\|_{M^{4,3}(P_{\frac{9}{10}})}\|{\partial_t}h_{k}\|_{M^{2,3}(P_{\frac{9}{10}})}\\
&\rightarrow 0.
\end{align*}
In the same way, one has
\begin{align*}
R_{3}
&=\lambda_{k}r_{k}\int_{P_{\frac{9}{10}}} h_{k}\times[\boldsymbol{v}(x_{k}+xr_{k},t_{k}+tr^{2}_{k})\cdot \nabla]h_{k}\cdot \zeta^{2}(h_{k}-h)dxdt\\
&\leq \lambda_{k}r_{k} \|\zeta\|^{2}_{L^{\infty}(P_{\frac{9}{10}})}\|h_{k}-h\|_{M^{4,3}(P_{\frac{9}{10}})}
\|h_{k}\|_{M^{4,3}(P_{\frac{9}{10}})}\|\nabla h_{k}\|_{M^{2,3}(P_{\frac{9}{10}})}\|\boldsymbol{v}\|_{L^{\infty}(P_{\frac{9}{10}})}\\
&\rightarrow 0.
\end{align*}

Next, we consider the term $R_{1}$. Note that
\begin{align*}
R_{1}
&=\lambda_{k}\int_{P_{\frac{9}{10}}}\zeta^{2}h_{k,x_{l}}^{j}\left(h_{k,x_{l}}^{j}(a_{k}^{i}+\lambda_{k}h_{k}^{i})(h_{k}^{i}-h^{i})\right)dxdt\\
&=\lambda_{k}\int_{P_{\frac{9}{10}}}\zeta^{2}h_{k,x_{l}}^{j}\left(h_{k,x_{l}}^{j}(a_{k}^{i}+\lambda_{k}h_{k}^{i})
-h_{k,x_{l}}^{i}(a_{k}^{j}+\lambda_{k}h_{k}^{j})\right)(h_{k}^{i}-h^{i})dxdt,
\end{align*}
where $i,j,l=1,2,3$, and the last equality holds in the light of $|a_{k}+\lambda_{k}h_{k}|=1$ (Here is H\'{e}lein's trick \cite{Helein}). Write $b_{k,l}^{i,j}=h_{k,x_{l}}^{j}(a_{k}^{i}+\lambda_{k}h_{k}^{i})
-h_{k,x_{l}}^{i}(a_{k}^{j}+\lambda_{k}h_{k}^{j})$. As $k\rightarrow \infty$, by H\"{o}lder's inequality, Young's inequality, the Plancherel equality, the Sobolev inequality, \eqref{weak conver}, \eqref{strong conver} and \eqref{hk-L4-strong}, one has
\begin{align*}
R_{1}
&=\lambda_{k}\int_{P_{\frac{9}{10}}}\zeta h_{k,x_{l}}^{j}b_{k,l}^{i,j}\zeta(h_{k}^{i}-h^{i})dxdt\\
&\leq\lambda_{k}\|h_{k,x_{l}}^{j}\zeta b_{k,l}^{i,j}\|_{M^{\frac{3}{2},3}(P_{\frac{9}{10}})}\|\zeta(h_{k}^{i}-h^{i})\|_{M^{3,3}(P_{\frac{9}{10}})}\\
&\leq \lambda_{k}\|h_{k,x_{l}}^{j}\|_{M^{2,3}(P_{\frac{9}{10}})}
\|\zeta b_{k,l}^{i,j}\|_{M^{6,3}(P_{\frac{9}{10}})}\|h_{k}^{i}-h^{i}\|_{M^{4,3}(P_{\frac{9}{10}})}\|\zeta\|_{M^{12,3}(P_{\frac{9}{10}})}\\
&\leq Co(1)\|\zeta b_{k,l}^{i,j}\|_{M^{6,3}(P_{\frac{9}{10}})}\\
&\leq Co(1)\|\nabla (\zeta b_{k,l}^{i,j})\|_{M^{2,3}(P_{\frac{9}{10}})}\\
&\leq Co(1)\|i\xi\hat{\zeta}(\xi)\ast\hat{b_{k,l}^{i,j}}(\xi)\|_{M^{2,3}_{\xi}(\mathbb{R}^{3})}\\
&\leq Co(1)\|i\xi\hat{\zeta}(\xi)\|_{M^{1,3}(\mathbb{R}^{3})}\|b_{k,l}^{i,j}\|_{M^{2,3}(P_{\frac{9}{10}})}\\
&\leq Co(1)\|\nabla \zeta\|^{2}_{M^{2,3}(P_{\frac{9}{10}})}\|b_{k,l}^{i,j}\|_{M^{2,3}(P_{\frac{9}{10}})}\\
&\rightarrow 0.
\end{align*}
Hence, we get that $R_{1}+R_{2}+R_{3}\rightarrow0$. Therefore,  we have
\begin{align*}
\int_{P_{\frac{1}{2}}}|\nabla h_{k}-\nabla h|^{2}dxdt\leq o(1), \text{ as } k\rightarrow\infty.
\end{align*}
Obviously, we get the desired result.
\end{proof}

\begin{lemma}[\cite{Giaq}, Lemma 6]\label{Lpa}
If $\mu>1$, then $\mathscr{L}^{p,\mu}(A,\delta)$ is isomorphic to $C^{0,\alpha}(\Omega,\delta)$, the space of $\alpha$-H\"{o}lder continuous functions with respect to the metric $\delta$, where $\alpha=(\frac{n+2}{p})(\mu-1)$, $A=\Omega\times(0,T)$ is an open set in $\mathbb{R}^{n+1}$, $\Omega$ is open in $\mathbb{R}^{n}$, and $\mathscr{L}^{p,\mu}(A,\delta), \mu>0$ is the Banach space of all functions in $L_{p}$ satisfying
\begin{align*}
[f]^{p}_{p,\mu,A}=\sup_{z_{0}\in A,r>0}[meas(A\cap Q(z_{0},r))]^{-\mu}\int_{A\cap Q(z_{0},r)}|f-{\int\hspace{-1.05em}-}_{z_{0},r}f|^{p}dx<\infty,
\end{align*}
and ${\int\hspace{-0.90em}-}_{z_{0},r}f=\frac{1}{meas(A\cap Q(z_{0},r))}\int_{A\cap Q(z_{0},r)}f(z)dz$.
\end{lemma}

Inspired by \cite[Section 6]{Feld}, we see that the key to improve the regularity is that the solution is bounded in $C^{0,\alpha}(P_{r}(z_{0}))$ and $E(z,\rho)<\rho^{2\gamma}$.
\begin{proposition}\label{holder space}
If we have $\boldsymbol{m}\in C^{0,\alpha}(P_{r}(z_{0}))$ and $E(z,\rho)<\rho^{2\gamma}$, then $\boldsymbol{m}\in C^{\infty}(P_{\frac{r}{2}}(z_{0}))$.
\end{proposition}

\begin{proof}
Firstly, we want to show that $E(z,\rho)<\rho^{2\gamma}$. Assume that $2c_{1}>1, \tau\in(0,(2c_{1})^{-\frac{1}{2}})$, $z_{0}=(x_{0},t_{0}), t_{0}>T_{0}$ and $E(z_{0},r)<\varepsilon_{0}(\frac{T_{0}}{2},\tau)$. Iterating \eqref{E(z,r)}, we can obtain
$E(z_{0},\tau^{k}r)\leq (2c_{1}\tau^{2})^{k}E(z_{0},r)$. We see that $k\rightarrow\infty$ is equivalent to:
\begin{align*}
\lim_{r\rightarrow0^{+}}E(z_{0},r)=0.
\end{align*}
By \cite[Theorem 1]{Giaq}, we know that there exists constants $r>0,~s<\frac{r}{2}$ such that for any $\gamma\in(0,1),~z\in P_{s}(z_{0})$, $\rho<\frac{r}{2}$, then one has
\begin{align*}
E(z,\rho)<c\rho^{2\gamma}.
\end{align*}

Next, we want to prove $\boldsymbol{m}\in C^{0,\alpha}(P_{r}(z_{0}))$. Notice that we can prove $\boldsymbol{m}\in C^{0,\alpha}(P_{r}(z_{0}))$ instead of proving $\boldsymbol{m}\in \mathscr{L}^{p,\mu}(P_{r}(z_{0}))<\infty$ by Lemma \ref{Lpa}.
Applying Lemma \ref{lemma BMO es}, we have
\begin{align*}
\int_{p_{\alpha r(z)}}|\boldsymbol{m}(x,t)-(\boldsymbol{m})_{z,\alpha r}|^{2}dxdt
&\leq C_{\alpha}r^{2}\int_{p_{r(z)}}|\nabla \boldsymbol{m}|^{2}dxdt \\
&\leq C_{\alpha}r^{2+n}E(z,r)
\leq C_{\alpha}r^{2+n+2\gamma}.
\end{align*}
Selecting $\mu=\frac{2\gamma}{n+2}+1$, then one has
\begin{align*}
\frac{\int_{p_{\alpha r(z)}}|\boldsymbol{m}(x,t)-(\boldsymbol{m})_{z,\alpha r}|^{2}dxdt}
{|P_{\alpha r(z)}|^\mu}
\leq C_{\alpha}\frac{r^{2+n+2\gamma}}{r^{\mu(n+2)}}
= C_{\alpha}<\infty.
\end{align*}
By Lemma \ref{Lpa}, we can infer that $\boldsymbol{m}\in \mathscr{L}^{2,\frac{2\gamma}{n+2}+1}(P_{r}(z_{0}))<\infty$. Then, $\boldsymbol{m} \in C^{0,\gamma}(P_{r}(z_{0}))$. By the higher order regularity theorem \cite[Lemma 21]{Feld}, we conclude that $\boldsymbol{m}\in C^{\infty}(P_{r}(z_{0}))$. Obviously, we complete the proof of Proposition \ref{holder space} and then we obtain the main theorem.
\end{proof}

\smallskip

\section*{Acknowledgments}
H. Wang's research is supported by the National Natural Science Foundation of China (Grant No.~11901066), the
Natural Science Foundation of Chongqing (No.~cstc2019jcyj-msxmX0167) and projects No.~2022CDJXY-001,  No.~2020CDJQY-A040
supported by the Fundamental Research Funds for the Central Universities.

\end{document}